\documentclass[11pt]{extarticle}
\def\version{December 30, 2024}

\nonstopmode

\usepackage[normalem]{ulem}

\newcommand{\ac}[1]{\noindent\textcolor{red}
{{\rm [\![}\mbox{\sc{AC}$\blacktriangleright\!\!\blacktriangleright$}: {#1}{\rm ]\!]}}}

\newcommand{\nb}[1]{\noindent\textcolor{blue}
{{\rm [\![}\mbox{\sc{NB}}: {#1}{\rm ]\!]}}}

\renewcommand{\ac}[1]{}
\renewcommand{\nb}[1]{}

\newcommand{\e}{\bm{e}}

\newcommand{\ds}{\displaystyle}

\usepackage{comment}
\usepackage{color}
\definecolor{MyDarkBlue}{rgb}{0,0.08,0.45}
\usepackage[backref=none]{hyperref}
\hypersetup{pdfborder={0 0 0},
  colorlinks,
  urlcolor={MyDarkBlue},
  linkcolor={MyDarkBlue},
  citecolor={MyDarkBlue},
  breaklinks=true}
\usepackage{doi}
\providecommand{\eprint}[1]{}
\renewcommand{\eprint}[1]{arXiv:\href{http://arxiv.org/abs/#1}{#1}}

\usepackage{color}
\usepackage{latexsym}
\usepackage{bm}
\usepackage{upgreek}

\usepackage{mathrsfs}
\usepackage{times}
\usepackage{amsmath}
\usepackage{amsthm}
\usepackage{amssymb}
\usepackage{epsfig}
\usepackage{graphicx}

\newcommand{\jj}{\mathrm{i}}
\providecommand{\itref}[1]{{\it (\ref{#1})}}

\DeclareOldFontCommand{\brianup}{\upshape}{\mathrm}

\DeclareSymbolFont{EUR}{U}{eur}{m}{n}
\SetSymbolFont{EUR}{bold}{U}{eur}{b}{n}
\DeclareSymbolFontAlphabet{\eur}{EUR}

\DeclareSymbolFont{EUB}{U}{eur}{b}{n}
\SetSymbolFont{EUB}{bold}{U}{eur}{b}{n}
\DeclareSymbolFontAlphabet{\eub}{EUB}

\DeclareSymbolFont{AMSb}{U}{msb}{m}{n}
\DeclareSymbolFontAlphabet{\mathbb}{AMSb}

\newcommand\scrX{\mathscr{X}}
\newcommand\scrY{\mathscr{Y}}
\newcommand\frakX{\mathfrak{X}}
\newcommand\frakY{\mathfrak{Y}}

\newcommand{\hh}{\mbox{\sl h}}

\newcommand\calA{\mathcal{A}}

\newcommand\calL{\mathcal{L}}

\newcommand{\notyet}[1]{{}}

\newcommand{\dom}{\mathfrak{D}}

\newcommand{\p}{\partial}
\renewcommand{\P}{\grave{\partial}}

\newcommand{\at}[1]{\vert\sb{\sb{#1}}}

\def\R{\mathbb{R}}
\newcommand{\C}{\mathbb{C}}\newcommand{\Z}{\mathbb{Z}}
\newcommand{\N}{\mathbb{N}}

\newcommand{\abs}[1]{\vert #1 \vert}

\newcommand{\sothat}{{\rm ;}\ }

\DeclareMathSymbol{\varGamma}{\mathord}{letters}{"00}
\DeclareMathSymbol{\varDelta}{\mathord}{letters}{"01}
\DeclareMathSymbol{\varTheta}{\mathord}{letters}{"02}
\DeclareMathSymbol{\varLambda}{\mathord}{letters}{"03}
\DeclareMathSymbol{\varXi}{\mathord}{letters}{"04}
\DeclareMathSymbol{\varPi}{\mathord}{letters}{"05}
\DeclareMathSymbol{\varSigma}{\mathord}{letters}{"06}
\DeclareMathSymbol{\varUpsilon}{\mathord}{letters}{"07}
\DeclareMathSymbol{\varPhi}{\mathord}{letters}{"08}
\DeclareMathSymbol{\varPsi}{\mathord}{letters}{"09}
\DeclareMathSymbol{\varOmega}{\mathord}{letters}{"0A}

\theoremstyle{plain}
\newtheorem{theorem}{Theorem}[section]

\newtheorem{lemma}{Lemma}[section]

\theoremstyle{definition}

\theoremstyle{remark}
\newtheorem{remark}[lemma]{Remark}
\newtheorem{example}[lemma]{Example}

\renewcommand{\theequation}{\thesection.\arabic{equation}}

\makeatletter\@addtoreset{equation}{section}
\makeatletter\@addtoreset{theorem}{section}
\makeatother

\renewcommand{\Re}{\mathop{\rm{R\hskip -1pt e}}\nolimits}
\renewcommand{\Im}{\mathop{\rm{I\hskip -1pt m}}\nolimits}
\newcommand{\bfH}{\mathbf{H}}

\newcommand{\bfX}{\mathbf{X}}
\newcommand{\bmK}{\bm{K}}

\begin{document}
\renewcommand{\theequation}{\thesection.\arabic{equation}}
\newcommand{\sect}[1]{\setcounter{equation}{0}\section{#1}}

\title{On spectral stability
of one- and bi-frequency solitary waves in Soler model in (3+1)D
}

\author{
{\sc Nabile Boussa{\"\i}d}
\\
{\small\it
Universit\'e Marie et Louis Pasteur, CNRS, Institut UTINAM,
\'{e}quipe de physique th\'{e}orique},
\\
{\small\it
F-25000 Besan\c{c}on, France
}
\\~\\
{\sc Andrew Comech}
\\
{\small\it
Mathematics Department,
Texas A\&M University, College Station, TX 77843-3368, USA}
\\~\\
{\sc Niranjana Kulkarni}
\\
{\small\it
Mathematics Department,
Texas A\&M University, College Station, TX 77843-3368, USA}
}

\date{\version}

\maketitle

\begin{abstract}
For the nonlinear Dirac equation with scalar self-interaction (the Soler model) in three spatial dimensions,
we consider the linearization at solitary wave solutions and find the invariant spaces
which  correspond to different spherical harmonics,
thus achieving the radial reduction of the spectral stability analysis.
We apply the same technique to the bi-frequency solitary
waves (which are generically present in the Soler model)
and show that they can also possess linear stability properties
similar to those of one-frequency solitary waves.

\end{abstract}

\section{Introduction}

Properties of fermionic fields
are still somewhat of a mystery.
Pre-history of spinors starts only in the 19th century
(if we do not take into account
the figures from the candle dance of Southeast Asia):
it could be traced back to
William Hamilton's research on quaternions (1843).
Irreducible even-dimensional representations
of $\mathbf{SO}(3)$ were studied by
\'{E}lie Cartan \cite{cartan1913groupes}.
Two-component vectors
-- points of the space corresponding to the
irreducible representation
of $\mathbf{SO}(3)$ in $\C^2$ --
appeared in Wolfgang Pauli's theory
\cite{pauli1927quantenmechanik}
and were subsequently called
\emph{spinors} by Paul Ehrenfest
and studied in van der Waerden's
\emph{Spinoranalyse}
\cite{van-der-waerden1929spinoranalyse,van-der-waerden-1932}.
Four-component spinors
are at the foundation of the relativistic theory
of Paul Dirac \cite{dirac1928};
this theory led to correct values for the spin
and the magnetic moment of the electron
and subsequently
to the most accurate agreement of the spectrum of the Hydrogen
and the eigenvalues of the Dirac operator
in the Coulomb potential
(with the notable exception of the Lamb shift),
suggesting that the electrons are
to be described by the field
that takes values in the four-dimensional ``spinor''
space.\footnote{So nice an agreement is a delicate point:
the exact formula for the energy levels in the Dirac theory
was first obtained by Sommerfeld
in \cite{sommerfeld1916quantentheorie},
twelve years before Dirac,
via applying
a completely empirical Bohr--Sommerfeld quantization
taking into account the relativistic precession of
elliptic orbits
known from Astronomy.}

While it is said colloquially that
spinors ``have no classical analogues'',
some geometrical properties of our three-dimensional world
could be related to spinorial representations
(or, more precisely, to the fact that the spinor group
$\mathbf{Spin}(3)$
is a simply connected double cover of $\mathbf{SO}(3)$,
which in turn is connected but not simply connected).
These phenomena are closely related
to the idea of the M\"{o}bious strip,
such as
 the change of polarization of light
 when moving along a closed contour
 in a medium with a small diffraction index,
 when the polarization plane rotates by $\pi$
 (see e.g. \cite{monastyrskij1999ignatovsky}),
the ``Berry phase'' \cite{berry1984quantal} (also known as
``Pancharatnam--Berry phase'' or ``geometric phase'')
of an eigenvector corresponding to a simple eigenvalue,
which
can be $\pi n$
(with $n$ not necessarily even)
after the motion over a closed non-contractible contour,
or Dirac's string trick
(related to the candle-dance trick).

The Dirac theory of spinor fields
-- and the Atomic Physics in general --
still seems incomplete,
as Dirac himself was pointing out
in early eighties \cite{dirac1981does}.
In particular,
using the Dirac equation for more than one electron
is still problematic.
The problems appearing in the description of
two electrons in Helium
based on the Dirac equation are discussed in
\cite{kalf1981nonexistence},
\cite{derezinski2012open}
(the discussion goes as far as to conjecture
that the electrons in Helium are not in a bound state
but rather that they form a \emph{resonance}),
and then in \cite{okaji2014spectral}.
Let us also mention an attempt \cite{kiessling2016novel}
to interpret the Dirac field
as describing a more fundamental particle,
with electron and positron being
two of its different ``topological spin'' states;
see also the discussion in e.g. \cite{bialynicki2020comment}
on the presence of positronic component in the solution to the Dirac equation.
As the matter of fact, even our understanding
of a single electron seems far from complete.
Can we hope to have as its counterpart a localized solution
of the Dirac--Maxwell system~\cite{gross1966cauchy,wakano-1966}
in the framework of classical, non-quantized, interacting fields
-- are there stable localized solutions?
The quantum theory essentially operates with the plane waves
whose  self-interaction is dropped: physically, one says that the
renormalization has been performed,
so that the ``quanta'' of the fermionic field
have the mass and the charge known from the experiment.
While mathematics had nothing to offer to the physics demand
when the quantum electrodynamics was being developed,
now the mathematical background
for the treatment of the Dirac--Maxwell system of classical fields
is being formed:
the local well-posedness of Dirac--Maxwell
was proved in~\cite{bournaveas1996local};
the existence of standing waves
was demonstrated numerically in \cite{lisi1995solitary}
and proved in \cite{esteban1996stationary,abenda1998solitary}.
The question
of stability of localized classical solutions
of Dirac--Maxwell
(and thus of its possible relation to Physics)
can finally be consistently stated and discussed.

According to the physics folklore,
since the Dirac energy density
is not positive-definite,
even the zero solution
(``vacuum'')
of the Dirac--Maxwell cannot be stable
due to the continuous creation and annihilation
of electron-positron pairs;
the way to a stable electron is via the second quantization,
with all negative energy states of the ``Dirac sea'' having been filled
and with the Pauli exclusion principle
preventing the birth of more negative-energy electrons.
On the other hand,
owing to the Dirac equation being the first order in time,
the physical intuition based on Newton's second law (which is second order in time)
is not necessarily applicable;
see, e.g., \cite[Section V]{ranada1983classical}.
Indeed, using the cubic nonlinear Dirac equation \cite{jetp.8.260,PhysRevD.1.2766}
to model the self-interaction of the the Dirac--Maxwell system,
one finds that solitary waves
$\phi_\omega(x)e^{-\jj\omega t}$
in the limit nonrelativistic limit $\omega\lesssim m$,
with $m$ the electron's mass,
in spatial dimension one and two are spectrally stable
\cite{berkolaiko2012spectral,lakoba2018numerical,linear-b}.
According to \cite{PhysRevLett.116.214101},
solitary waves of the cubic Dirac equation in (3+1)D,
while spectrally unstable for $\omega\lesssim m$ close to $m$,
seem to become spectrally stable for $\omega\lesssim\omega_\ast$,
with $\omega_\ast\approx 0.936\,m$.
It is at this value of $\omega$
that the pair of real eigenvalues
-- one positive (the one responsible for
linear instability) and one negative,
present in the spectrum of linearization operator
for
$\omega_\ast<\omega<m$ -- collide
at the origin, producing two purely imaginary eigenvalues.
Computations of the spectrum of the linearization at solitary waves
have been performed in \cite{PhysRevLett.116.214101} only in the radial case;
no approach to consider higher degree spherical harmonics has been developed.
We are going to overcome this situation in the present work,
paving the way for the future numeric computation
of the spectrum of the linearization at solitary waves
in the Soler model in (3+1)D for perturbations of different angular structure
and hoping that this approach would also be
applicable to the physically viable Dirac--Maxwell.
Let us mention that consideration of perturbations
which are
of the same angular structure as the soliton itself
suggests -- but not proves --
the stability of solitons of Dirac--Maxwell
in the nonrelativistic limit
$\omega\gtrsim -m$;
see \cite{comech2013polarons,comech2018small}.

\smallskip

\noindent
{\bf Bi-frequency solitary waves.\ }
One peculiarity of the Soler model is the existence
of bi-frequency solitary waves,
which is a consequence of the $\mathbf{SU}(1,1)$ symmetry
of the Soler model:
\begin{align}\label{bo}
\psi\mapsto (a+b\gamma^2\bmK)\psi,
\qquad
\gamma^2=
\begin{bmatrix}0&\sigma_2\\-\sigma_2&0\end{bmatrix},
\qquad
a,\,b\in\C,
\quad
\abs{a}^2-\abs{b}^2=1,
\end{align}
where $\bmK:\,\C^4\to\C^4$ is the operator of complex conjugation.
We note that \eqref{bo} contains $\mathbf{U}(1)$ as its subgroup
(with $b=0$);
the generator $\jj\gamma^2\bmK$
is known as the charge conjugation
in Quantum Electrodynamics.
The symmetry transformation \eqref{bo}
coincides with the Bogoliubov transformation
\cite{bogolyubov1958}
from the solid state physics.
It seems that
the presence of this continuous symmetry
in the context of the Dirac equation
was first noticed in \cite{galindo1977remarkable}.
One can see that
the application of the Bogoliubov symmetry
\eqref{bo}
to a solitary wave $\phi_\omega(x)e^{-\jj\omega t}$
produces bi-frequency solitary waves
\cite{boussaid2018spectral}.

\begin{remark}
We point out that
while the Soler model seems to be merely a playground
for developing tools for the Dirac--Maxwell system,
with the symmetry \eqref{bo} nothing but a curious artifact
(which is absent in Dirac--Maxwell),
a physically viable Dirac--Klein--Gordon system
describing interaction or self-interaction of fermions
via the exchange of scalar particles
(``Yukawa coupling''),
with the Higgs boson an example of such a scalar particle,
features the same $\mathbf{SU}(1,1)$ symmetry,
and so does any spinor theory
with the interaction term
in the Hamiltonian density based on
the scalar quantity $\psi^*\beta\psi$.
The consequences of $\mathbf{SU}(1,1)$ symmetry,
such as the existence of bi-frequency solitary waves
and the conservation of a complex-valued charge
$\Lambda=\int_{\R^3}
\psi^*(t,x)(-\jj\gamma^2)\bmK\psi(t,x)\,dx$
(see \cite{boussaid2018spectral}),
remain in effect for all such theories.

It seems tempting yet still too early
to try to see a relation of bi-frequency localized solutions
in fermionic models with scalar self-interaction
to actual physical phenomena,
since the effects that may be connected to the
interaction via the Higgs mechanism
-- such as neuton--antineutron oscillations
\cite{phillips2016neutron}
or neutron--mirror neutron (the latter is considered
a hidden sector particle, related to the Dark Matter)
\cite{kamyshkov2022neutron,dvali2024kaluza}
--
are presently under active study in High Energy Physics
which tries to evolve beyond the Standard Model.
Other localized solutions -- or quasiparticles --
in the field theories
include polarons
(see \cite{alexandrov2008polarons,franchini2021polarons})
and skyrmions (see the reviews \cite{zahed1986skyrme,everschor-sitte2018perspective});
both have been found to be directly related to
different physical phenomena.
\end{remark}

If a particular bi-frequency solitary
wave can be obtained by the application of
$\mathbf{SU}(1,1)$ transformation
to a one-frequency solitary wave,
then its spectral stability --
understood in a proper way --
would be the same as that
of one-frequency solitary waves
(take a bi-frequency solitary wave with a small perturbation;
apply the inverse transformation to obtain a one-frequency solitary wave with a small perturbation).
This approach does not work in spatial dimensions three and higher,
since the solitary manifold of bi-frequency solitary waves
is \emph{larger} than the symmetry group $\mathbf{SU}(1,1)$
\cite{boussaid2018spectral}.
For example, in three spatial dimensions,
while
the solitary manifold of one-frequency solitary waves
\begin{align}\label{s1}
\mathscr{S}_1=\left\{
\begin{bmatrix}
v(r,\omega)\bm\xi\\\jj u(r,\omega)\sigma_r\bm\xi
\end{bmatrix}e^{-\jj\omega t},
\quad
\bm\xi\in\C^2,
\ \abs{\bm\xi}=1
\right\},
\end{align}
with $v(r,\omega),\,u(r,\omega)$ real-valued
and $\sigma_r=\abs{x}^{-1}x\cdot\sigma$,
is of (real) dimension $4$ (we consider $\omega$ as a parameter);
the group $\mathbf{SU}(1,1)$ is of dimension $3$,
sharing the one-dimensional $\mathbf{U}(1)$-symmetry with
\eqref{s1};
as a result, the set
$\big\{(a+b\gamma^2\bmK)\phi_\omega(x)e^{-\jj\omega t}\big\}$
is six-dimensional.
At the same time,
the dimension of the solitary manifold
of bi-frequency solitary waves,
\begin{align}\label{s2}
\mathscr{S}_2
=\left\{
\begin{bmatrix}v(r,\omega)\bm\xi\\\jj u(r,\omega)\sigma_r\bm\xi
\end{bmatrix}e^{-\jj\omega t}
+
\begin{bmatrix}-\jj u(r,\omega)\sigma_r\bm\eta\\v(r,\omega)\bm\eta
\end{bmatrix}e^{\jj\omega t},
\quad
\bm\xi,\,\bm\eta\in\C^2,
\ \abs{\bm\xi}^2-\abs{\bm\eta}^2=1
\right\},
\end{align}
equals eight.
Bi-frequency solitary waves that
can be obtained from one-frequency solitary waves with a symmetry transformation
\eqref{bo} are the ones with $\bm\eta$ parallel to $\sigma_2\bmK\bm\xi$.
We note that $\bm\xi$ and $\sigma_2\bmK\bm\xi$ are mutually orthogonal;
thus, the solitary waves
that cannot be obtained as an element of
$\mathbf{SU}(1,1)$ acting on a one-frequency solitary wave
are the ones with $\bm\xi$, $\bm\eta$ being non-orthogonal.
(We call solutions
of the form \eqref{s2}
\emph{bi-frequency}
rather than \emph{two-frequency} solitary waves
to emphasize the fact that
the frequencies $\pm\omega$ are mutually opposite.)

In four spatial dimensions,
the Soler model for four-component spinors
has $\mathbf{U}(1)$ symmetry
but not the $\mathbf{SU}(1,1)$ symmetry
as a result,
none of the bi-frequency solitary waves
\eqref{s2} with $\bm\eta\ne 0$
admits the reduction of stability analysis
to that of a one-frequency solitary wave.
At the same time,
a bi-frequency solitary wave from \eqref{s2},
with $\abs{\bm\xi}\gtrsim 1$ and $\bm\eta$ small,
can be considered as a small perturbation
of a one-frequency solitary wave from \eqref{s1};
this means that the asymptotic stability for one-frequency solitary waves
is simply not possible: a small perturbation,
in the form of the bi-frequency solitary wave,
being an exact solution, would never relax to a one-frequency solitary wave.
This brings us to the need to perform the stability analysis
of bi-frequency solitary waves directly, including the cases
when the reduction to a one-frequency solitary wave is not possible.
An attempt at this stability analysis has been made in
\cite{boussaid2018spectral},
but that approach seems erroneous for spatial dimensions
$3$ and $4$
and is certainly useless
for potential applications to asymptotic stability.

In the present work, we develop a simple approach to the spectral stability
of bi-frequency solitary waves
in spatial dimensions three,
based on the radial reduction,
showing that the spectral stability of one-
and bi-frequency solitary waves is different.
As the matter of fact,
bi-frequency solitary waves
seem to have stability properties
which are the same or better
than those of one-frequency solitary waves;
see Remark~\ref{remark-better}.

The article is organized as follows.
In Section~\ref{sect-soler},
we describe the model and give the main notations.
In Section~\ref{sect-soler-one},
we perform the radial reduction
for the linear stability analysis
of one-frequency solitary waves
via decomposition into invariant spectral subspaces
corresponding to different spherical harmonics.
In Section~\ref{sect-soler-bi},
we generalize this approach
to linearization at a bi-frequency solitary waves
and in Section~\ref{sect-init}
show that we considered all perturbations.
In Appendix~\ref{sect-harmonics},
we provide our notations for the spherical harmonics.
In Appendix~\ref{sect-spin-orbit},
we derive the needed properties
of the spin-orbit operator in the Dirac theory.

\bigskip

\noindent
{\sc Acknowledgements.}
This research has been funded in whole or in part by the French National Research Agency (ANR) as part of the QuBiCCS project ``ANR-24-CE40-3008-01''.

This work was supported by a grant from the Simons Foundation (851052, A.C.).

For the purpose of its open access publication, the author/rights holder applies a CC-BY open access license to any article/manuscript accepted for publication (AAM) resulting from that submission.

\section{The Soler model}
\label{sect-soler}

The Soler model
has the form
\begin{equation}\label{Soler}
\jj\p\sb{t}\psi
=D_m\psi-f(\psi^*\beta\psi)\beta\psi
=-\jj\bm\alpha\cdot\nabla\psi+
(m-f(\psi^*\beta\psi))\beta\psi,
\quad
\psi(x,t)\in\C^4,
\ \ x\in\R^3,
\end{equation}
where
$\alpha^i$ ($1\le i\le 3$)
and $\beta$
are self-adjoint Dirac matrices
of size $N=4$;
$
D\sb m
=
\bm\alpha\cdot\bm{p}+\beta m
=-\jj\bm\alpha\cdot\nabla+\beta m
$, $m>0$,
is the Dirac operator.
The self-interaction in \eqref{Soler}
is represented by a real-valued function
$f$, with $f(0)=0$;
the original Soler model
\cite{jetp.8.260,PhysRevD.1.2766}
is cubic, so that $f(\tau)=\tau$, $\tau\in\R$.
Below, we will use the notation
\begin{align}\label{def-g}
g(\tau)=m-f(\tau),
\qquad
\tau\in\R.
\end{align}
The Dirac matrices are self-adjoint,
anticommuting, with their square being equal to one,
so that
\[
D_m^2=(-\Delta^2+m^2)I_4,
\]
with $I_4$ is the unit matrix of size $N$.
It then follows that $\alpha^j$ and $\beta$
satisfy the standard relations
\begin{align}\label{anti-alpha}
\alpha_i\alpha_j
+\alpha_j\alpha_i=2\delta_{i j} I_4,
\qquad
\alpha_i\beta
+\beta\alpha_i=0,
\qquad
1\le i,\,j\le 3;
\qquad
\beta^2=I_4.
\end{align}
The Dirac conjugate of $\psi\in\C^4$
is denoted by
$\bar\psi=\psi^*\beta$,
with $\psi^*$ the hermitian conjugate of $\psi$.

The Dirac matrices could be taken in the form
\begin{equation}\label{upalpha}
\alpha^{j}
=
\begin{bmatrix} 0&\sigma_j
\\ \sigma_j&0\end{bmatrix},
\qquad
\beta
=
\begin{bmatrix} I_2&0 \\ 0&-I_2\end{bmatrix},
\end{equation}
with
$\sigma_j$,
$1\le j\le 3$,
the Pauli matrices.
For $x\in\R^3\setminus\{0\}$, $r=\abs{x}>0$, we denote
\[
\sigma_r=r^{-1}x\cdot\bm\sigma,
\qquad
\alpha_r=r^{-1}x\cdot\bm\alpha,
\]
and we introduce
\[
\varSigma=\bm\sigma\cdot\nabla:\;L^2(\R^3,\C^2)\to L^2(\R^3,\C^2),
\qquad
\dom(\varSigma)=H^1(\R^3,\C^2),
\]
and let
$
\varSigma_\Omega:\,L^2(\mathbb{S}^2,\C^2)
\to L^2(\mathbb{S}^2,\C^2)$,
$\ \dom(\varSigma_\Omega)=H^1(\mathbb{S}^2,\C^2)$
be the angular part of $\varSigma$
defined by the relation
\[
\varSigma=\sigma_r\p_r+r^{-1}\varSigma_\Omega.
\]
Here is the Dirac operator in spherical coordinates:
\begin{equation}\label{dirac-spherical}
D\sb 0
=-\jj\bm\alpha\cdot\nabla
=-\jj\left(
\alpha\sb r\p\sb r
+r^{-1}\alpha\sb\phi\p\sb\phi
+r^{-1}\alpha\sb\theta\p\sb\theta
\right),
\end{equation}
$x\sp{1}=r\cos\phi\sin\theta$,
$x\sp{2}=r\sin\phi\sin\theta$,
$x\sp{3}=r\cos\theta$,
$\phi\in[0,2\pi)$,
$\theta\in [0,\pi]$,
$\alpha\sb{r}
=\begin{bmatrix}0&\sigma\sb{r}\\
\sigma\sb{r}&0\end{bmatrix}$,
$\alpha\sb{\phi}
=
\begin{bmatrix}0&\sigma\sb{\phi}\\
\sigma\sb{\phi}&0\end{bmatrix}$,
$\alpha\sb{\theta}
=
\begin{bmatrix}0&\sigma\sb{\theta}\\
\sigma\sb{\theta}&0\end{bmatrix}$,
with
\[
\sigma\sb{r}=
\begin{bmatrix}
\cos\theta&e^{-\jj\phi}\sin\theta\\
e^{\jj\phi}\sin\theta&-\cos\theta
\end{bmatrix},
\quad
\sigma\sb{\phi}=
\frac{1}{\sin\theta}
\begin{bmatrix}
0&-\jj e^{-\jj\phi}\\
\jj e^{\jj\phi}&0
\end{bmatrix},
\quad
\sigma\sb{\theta}=
\begin{bmatrix}
-\sin\theta&e^{-\jj\phi}\cos\theta\\
e^{\jj\phi}\cos\theta&\sin\theta
\end{bmatrix}.
\]
We define the operator
$
\varSigma\sb{\Omega}$
by the relation
$\varSigma=\sigma_r\p\sb r+r^{-1}\varSigma\sb{\Omega}$,
with
$
\varSigma
=\bm\sigma\cdot\nabla
=\sum\sb{j=1}\sp{3}\sigma\sb j\p\sb j
=\sigma\sb r\p\sb r
+r^{-1}\sigma\sb\phi\p\sb\phi
+r^{-1}\sigma\sb\theta\p\sb\theta
$;
that is,
$\varSigma_\Omega=\sigma_\phi\p_\phi+\sigma_\theta\p_\theta$.
Then
\begin{align}\label{sigma-Sigma-3d}
\varSigma\sb{\Omega}
\circ
\sigma\sb r
=2I_2
+\frac{\jj}{\sin\theta}\sigma\sb\theta\p\sb\phi
-\jj\sin\theta\,\sigma\sb\phi\p\sb\theta,
\qquad
\sigma\sb r\varSigma\sb{\Omega}
=
-\frac{\jj}{\sin\theta}\sigma\sb\theta\p\sb\phi
+\jj\sin\theta\,\sigma\sb\phi\p\sb\theta,
\end{align}
so that
$\{\varSigma\sb{\Omega},\sigma\sb r\}
=2 I_2$
(cf. Lemma~\ref{lemma-sigma}).
In the explicit form,
\begin{align}\label{sigma-Sigma-3d-explicit}
\sigma_r\varSigma
=
\p_r+r^{-1}\sigma_r\varSigma_\Omega
=
\p_r
-\frac{\jj}{r\sin\theta}
\begin{bmatrix}
-\sin\theta&e^{-\jj\phi}\cos\theta\\
e^{\jj\phi}\cos\theta&\sin\theta
\end{bmatrix}
\p\sb\phi
+
\frac{1}{r}
\begin{bmatrix}
0& e^{-\jj\phi}\\
-e^{\jj\phi}&0
\end{bmatrix}
\p\sb\theta.
\end{align}

\begin{lemma}\label{lemma-same-degree}
Let $\hh_{\ell,k}$ be a spherical harmonic of degree $\ell\in\N_0$
and order $k\in\N_0$, $-\ell\le \,k\le\ell$.
For any vectors $\bm\xi,\,\bm\eta\in\C^2\setminus\{0\}$,
the function
$h_{\ell,k}=\bm\eta^*\sigma_r\varSigma_\Omega \hh_{\ell,k}\bm\xi
\in C^\infty(\mathbb{S}^2)$
is a linear combination of spherical harmonics
of the same degree $\ell$:
there are coefficients
$C_{\ell,m,k}(\bm\xi,\bm\eta)\in\C$,
$-\ell\le m,\,k\le\ell$,
$k-1\le m\le k+1$,
such that
\[
\bm\eta^*\sigma_r\varSigma_\Omega \hh_{\ell,k}\bm\xi
=\sum\sb{m=\max(k-1,-\ell)}^{\min(k+1,\ell)}
\hh_{\ell,m}C_{\ell,m,k}(\bm\xi,\bm\eta).
\]
\end{lemma}

\begin{proof}
The operator
$\sigma_r\varSigma_\Omega
=r\sigma_r\,\bm\sigma\cdot\nabla-r\p_r$
commutes with
$S_0:=r\p_r-r\sigma_r\,\bm\sigma\cdot\nabla+(n-1)/2$
(here $n=3$)
and hence with
$\Delta_\Omega=(\frac{n-2}{2})^2-(S_0-1/2)^2$
(this follows from Lemma~\ref{lemma-s}
applied to
$S_0=r\p_r-r\sigma_r\,\bm\sigma\cdot\nabla+(n-1)/2$ 
as in \eqref{def-s},
with $n=3$ and with the self-adjoint matrices
$\sigma_i$, $1\le i\le 3$,
in place of $\alpha^i$).
Therefore,
$\Delta_\Omega
\big(\bm\eta^*\sigma_r\varSigma_\Omega \hh\bm\xi\big)
=
\bm\eta^*
\sigma_r\varSigma_\Omega
\Delta_\Omega
\hh\bm\xi
=
-\varkappa_{\ell}
\bm\eta^*\sigma_r\varSigma_\Omega \hh\bm\xi$,
showing that
$\bm\eta^*\sigma_r\varSigma_\Omega \hh\bm\xi$
is also a spherical harmonic of degree $\ell$.
The conclusion that
$C_{m,k}\ne 0$ only if $k-1\le m\le k+1$
follows from \eqref{sigma-Sigma-3d-explicit}
(which only contains
terms with no $\phi$-dependence
or
terms with $e^{\pm\jj\phi}$)
\end{proof}

\begin{lemma}\label{lemma-mssm}
We have:
\begin{align}\label{mssm-1}
&
\e_1^*\sigma_r\varSigma_\Omega \hh_{\ell,m}\e_1=-m \hh_{\ell,m},
\qquad
\e_2^*\sigma_r\varSigma_\Omega \hh_{\ell,m}\e_2=m \hh_{\ell,m};
\\
\label{mssm-2}
&
\e_2^*\sigma_r\varSigma_\Omega \hh_{\ell,m}\e_1=m \hh_{\ell,m} e^{-\jj\phi}\frac{\cos\theta}{\sin\theta}+ \frac{e^{\jj\phi}}2
\left[ \sqrt{\ell+m}\sqrt{\ell-m+1}\hh_{\ell,m-1} - \sqrt{\ell-m}\sqrt{\ell+m+1}\hh_{\ell,m+1}\right];\nonumber
\\
&\e_1^*\sigma_r\varSigma_\Omega \hh_{\ell,m}\e_2=m \hh_{\ell,m}e^{-\jj\phi}\frac{\cos\theta}{\sin\theta}-\frac{e^{-\jj\phi}}2
\left[ \sqrt{\ell+m}\sqrt{\ell-m+1}\hh_{\ell,m-1} - \sqrt{\ell-m}\sqrt{\ell+m+1}\hh_{\ell,m+1}\right];\nonumber
\\
&
\e_1^*\sigma_r\varSigma_\Omega \hh_{\ell,-\ell}\e_2=0,
\quad\qquad\qquad
\e_2^*\sigma_r\varSigma_\Omega \hh_{\ell,\ell}\e_1=0,
\end{align}
for $\ell\in\N_0$, $m\in\N_0$ and $-\ell\le m\le \ell$.
\end{lemma}

Above, $\e_1$ and $\e_2$
denote the standard basis vectors in $\C^2$.

\begin{proof}
The statement
\eqref{mssm-1}
follows
from the relation
(cf. \eqref{sigma-Sigma-3d-explicit})
\begin{equation}\label{Eq:Important}
\sigma_r\varSigma_\Omega \hh_{\ell,m}
=\frac{1}{\sin\theta}
\begin{bmatrix}
-\sin\theta&e^{-\jj\phi}\cos\theta
\\e^{\jj\phi}\cos\theta&\sin\theta
\end{bmatrix}
m \hh_{\ell,m}
+
\begin{bmatrix}
0&e^{-\jj\phi}\\-e^{\jj\phi}&0
\end{bmatrix}
\p_\theta \hh_{\ell,m}
\end{equation}
\nb{Using $\sqrt{1-x^2}\frac{d}{dx}{P_\ell^m}(x) = \frac12 \left[ (\ell+m)(\ell-m+1)P_\ell^{m-1}(x) - P_\ell^{m+1}(x)\right]$
the last term become explicit, no? This provides a more complete proof.
Lemma~\ref{lemma-same-degree} also becomes shorter?
The sequel can be included if we use the $P_\ell^m=0$ for $|m|>\ell$, no?}
and
\begin{align*}
\p_\theta  \hh_{\ell,m}
(\theta,\phi)
&=
-\frac12
\left[ \sqrt{\ell+m}\sqrt{\ell-m+1}\hh_{\ell,m-1} - \sqrt{\ell-m}\sqrt{\ell+m+1}\hh_{\ell,m+1}\right].
\end{align*}
For \eqref{mssm-2},
we notice that
$\hh_{\ell,-\ell}=c_{\ell}
\sin^\ell\theta e^{-\jj\ell\phi}$
(with some $c_{\ell}\in\R$)
hence, for $\ell\ge 1$,
\[
\sigma_r\varSigma_\Omega
\sin^\ell\theta e^{-\jj\ell\phi}
\e_2
=
\begin{bmatrix}
e^{-\jj\phi}\cos\theta
\\\sin\theta
\end{bmatrix}
(-\ell)
\sin^{\ell-1}\theta e^{-\jj\ell\phi}
+
\begin{bmatrix}
e^{-\jj\phi}\\0
\end{bmatrix}
\ell\sin^{\ell-1}\theta\cos\theta e^{-\jj\ell\phi}
=
(-\ell)
\sin^\ell\theta e^{-\jj\ell\phi}\e_2.
\]
The first relation in \eqref{mssm-2} follows;
the second is proved similarly.
\end{proof}

\section{Invariant subspaces
for linearization at one-frequency solitary waves}
\label{sect-soler-one}

Given a solitary wave
\begin{align}\label{sw}
\phi(x)e^{-\jj\omega t},
\qquad
\phi(x)
=
\begin{bmatrix}
v(r,\omega)\e_1
\\
\jj\sigma_r u(r,\omega)\e_1
\end{bmatrix}
e^{-\jj\omega t}
,
\end{align}
with $v,\,u$ satisfying the system
\[
\begin{cases}
\omega v=\p_r u+\frac{2}{r}u+g(v^2-u^2)v,
\\
\omega u=-\p_r v-g(v^2-u^2)u,
\end{cases}
\qquad
r>0
\]
(see e.g.
\cite{esteban1995stationary,boussaid2017nonrelativistic}),
we consider a perturbation
\begin{align}\label{swp}
\big(
\phi(x)+
\rho(t,x)
\big)e^{-\jj\omega t},
\qquad
\rho(t,x)\in\C^4.
\end{align}
The linearization at $\phi e^{-\jj\omega t}$
-- that is, the linearized equation on
$\rho$
-- takes the form
\begin{align}\label{lin}
\jj\p_t\rho
=\calL\rho
:=D_0\rho+g\beta\rho+2g'\beta\phi\Re(\phi^*\beta\rho)
-\omega\rho,\qquad
\dom(\calL)=H^1(\R^3,\C^4),
\end{align}
where both $g=g(\tau)$ and $g'=g'(\tau)$
are evaluated at $\tau=\phi^*\beta\phi
=v^2-u^2$.
We note that the operator $\calL$
in \eqref{lin}
is not $\C$-linear
because of the term $\Re(\phi^*\beta\rho)$.

We also introduce a self-adjoint operator
\begin{align}\label{def-l0}
\calL_0=D_0+g\beta-\omega I,
\qquad
\dom(\calL_0)=H^1(\R^3,\C^4).
\end{align}


\begin{lemma}
\label{lemma-lm}
\begin{enumerate}
\item
\label{lemma-lm-1}There are the following
invariant subspaces
of the linearization operator $\calL$:
\nb{The spaces are rewritten so as $A$, $B$, ... and $S$ are in $L^2$}
\begin{align*}
\scrX_{\ell,0}
&=
\left\{
\begin{bmatrix}
(A_0+B_0\sigma_r\varSigma)\hh_{\ell,0}\e_1
\\
\jj\sigma_r(P_0+Q_0\sigma_r\varSigma)\hh_{\ell,0}\e_1
\end{bmatrix}
\mid A_0,\,B_0,\,P_0,\,Q_0\in L^2(\R^+, r\,\mathrm{d}r)
\right\},
\\
\scrX_{\ell,m}
&=
\left\{
\sum\limits\sb{\pm}
\begin{bmatrix}
(A_{\pm m}+B_{\pm m}\sigma_r\varSigma)\hh_{\ell,\pm m}\e_1
\\
\jj\sigma_r(P_{\pm m}+Q_{\pm m}\sigma_r\varSigma)\hh_{\ell,\pm m}\e_1
\end{bmatrix}
\mid A_{\pm m},\,B_{\pm m},\,P_{\pm m},\,Q_{\pm m}\in L^2(\R^+, r\,\mathrm{d}r)
\right\},
\\
\scrY_{\ell}
&=
\Biggl\{
\begin{bmatrix}
R_{\ell}\hh_{\ell,-\ell}\e_2
\\
\jj\sigma_r S_{\ell}\hh_{\ell,-\ell}\e_2
\end{bmatrix}
\mid R_{\ell},\,S_{\ell} \in L^2(\R^+, r\,\mathrm{d}r)
\Biggr\},
\end{align*}
where
$\ell\in\N_0$,
$m\in\N$
and
$m\le\ell$.
\item
\label{lemma-lm-2}
There is the inclusion
$\sigma\big(-\jj\calL\at{\scrY_{\ell}}\big)\subset\jj\R$.
\end{enumerate}
\end{lemma}

\begin{proof}
\nb{Do we do the proof with such functions smooth and localized and identify the domain in terms of the coefficients $A$, $B$, ....?}
Clearly, multiplication by a scalar or by $\beta$
is invariant in $\scrX_{\ell,m}$ and in $\scrY_{\ell}$.
To prove the invariance of $D_0$
in $\scrX_{\ell,m}$ and in $\scrY_{\ell}$,
we compute:
\begin{align}
\label{mis}
&
D_0
\begin{bmatrix}
(A+B\sigma_r\varSigma)\hh\e_1
\\
\jj\sigma_r(P+Q\sigma_r\varSigma)\hh\e_1
\end{bmatrix}
=
-\jj
\begin{bmatrix}0&\varSigma\\\varSigma&0\end{bmatrix}
\begin{bmatrix}
(A+B\sigma_r\varSigma)\hh\e_1
\\
\jj\sigma_r(P+Q\sigma_r\varSigma)\hh\e_1
\end{bmatrix}
=
\begin{bmatrix}
\varSigma
\big(\sigma_r P+Q\varSigma\big)\hh
\e_1
\\
-\jj\varSigma(A+B\sigma_r\varSigma)\hh\e_1
\end{bmatrix}
\nonumber
\\[1ex]
&
\quad
=
\begin{bmatrix}
\big(
P'
+P\varSigma\circ\sigma_r
+Q'\sigma_r\varSigma
+Q\Delta
\big)\hh
\e_1
\\
-\jj(
A'\sigma_r
+A\varSigma
+B'\varSigma
+B\varSigma\sigma_r\varSigma
)\hh\e_1
\end{bmatrix}
=
\begin{bmatrix}
\big(
P'+\frac{2}{r}P-P\sigma_r\varSigma
+
Q'\sigma_r\varSigma+Q\Delta\big)\hh
\e_1
\\
-\jj(A'\sigma_r
+A\varSigma
+B'\varSigma
-B\sigma_r\Delta
)\hh\e_1
\end{bmatrix}
\nonumber
\\[1ex]
&
\quad
=
\begin{bmatrix}
\big(
P'+\frac{2}{r}P-\frac{\varkappa_{\ell}}{r^2}Q
+(Q'-P)\sigma_r\varSigma
\big)\hh
\e_1
\\
-\jj\Big(
 \sigma_r (A'+\frac{\varkappa_{\ell}}{r^2}B)
 +(A+B')\varSigma
\Big)\hh\e_1
\end{bmatrix},
\end{align}
where
$\varkappa_{\ell}:=\ell(\ell+1)$,
$\ell\in\N_0$,
are eigenvalues of the Laplace--Beltrami
operator on the sphere $\mathbb{S}^2$.
In the last two relations,
we used 
Lemma~\ref{lemma-sigma}
and Lemma~\ref{lemma-identities}.
This shows that
\nb{Don't we need to treat $\sigma_r\varSigma\hh\e_1$ before concluding?}
$D_0\scrX_{\ell,m}\subset\scrX_{\ell,m}$
(for $0\le m\le\ell$)
and
\nb{Is the following proved?}
$D_0\scrY_{\ell}\subset\scrY_{\ell}$.

Let us consider the term
$2g'\beta\phi\Re(\phi^*\beta\rho)$
in equation \eqref{lin}
for $\rho\in\scrX_{\ell,m}$,
$0\le m\le\ell$.
By Lemma~\ref{lemma-mssm},
one has
$\e_1^*\sigma_r\varSigma \hh_{\ell,\pm m}\e_1
=\mp r^{-1}m \hh_{\ell,\pm m}$;
therefore,
\[
\begin{bmatrix}
v\e_1
\\
\jj u \sigma_r\e_1
\end{bmatrix}^*
\beta
\sum\sb\pm
\begin{bmatrix}
(A_{\pm m}+B_{\pm m}\sigma_r\varSigma)\hh_{\ell,\pm m}\e_1
\\
\jj\sigma_r(P_{\pm m}+Q_{\pm m}\sigma_r\varSigma)\hh_{\ell,\pm m}
\e_1
\end{bmatrix}
=
\sum\sb\pm
\Big(
v A_{\pm m}-u P_{\pm m}\mp\frac{m}{r}(v B_{\pm m}-u Q_{\pm m})
\Big)\hh_{\ell,\pm m}.
\]
Taking into account that
$\bar \hh_{\ell,m}=\hh_{\ell,-m}$,
one derives:
\begin{align}\label{2re}
2\Re(\phi^*\beta\rho)
&=
\sum_\pm
\Big\{
\Big(
v A_{\pm m}-u P_{\pm m}
\mp\frac{m}{r}(v B_{\pm m}-u Q_{\pm m})
\Big)\hh_{\ell,\pm m}
\nonumber
\\
&
\qquad\qquad
+
\Big(
v\bar A_{\pm m}-u\bar P_{\pm m}
\mp\frac{m}{r}(v\bar B_{\pm m}-u \bar Q_{\pm m})
\Big)\hh_{\ell,\mp m}
\Big\}.
\end{align}
This shows that
$2g'\beta\phi\Re(\phi^*\beta\rho)\in\scrX_{\ell,m}$,
completing the proof of
Part~\itref{lemma-lm-1}.
\nb{What about $m=0$ and $\scrY_{\ell}
$?}

\nb{Maybe the following can be put in a remark?}
For future reference,
using \eqref{2re} and introducing the notation
\begin{equation}\label{def-w}
W=g'\begin{bmatrix}v^2&-uv\\-uv&u^2
\end{bmatrix},
\end{equation}
we can write:
\begin{align}\label{wcw}
2g'\begin{bmatrix}v\\-u\\0\\0
\end{bmatrix}
\Re(\phi^*\beta\rho)
=
\sum_\pm
\begin{bmatrix}W&\mp\frac{m}{r}W\\0&0
\end{bmatrix}
\left(
\begin{bmatrix}A_{\pm m}\\P_{\pm m}\\B_{\pm m}\\Q_{\pm m}
\end{bmatrix}
\hh_{\ell,\pm m}
+
\begin{bmatrix}\bar A_{\pm m}\\\bar P_{\pm m}\\\bar B_{\pm m}\\\bar Q_{\pm m}
\end{bmatrix}\hh_{\ell,\mp m}
\right)
\nonumber
\\
=
\sum_\pm
\begin{bmatrix}W&\mp \frac{m}{r}W\\0&0
\end{bmatrix}
\left(
\Psi_{\pm m}
\hh_{\ell,\pm m}
+
\bar\Psi_{\pm m}
\hh_{\ell,\mp m}
\right)
,\end{align}
where
$\Psi_{\pm m}(t,r)
=\big(A_{\pm m}(t,r),P_{\pm m}(t,r),B_{\pm m}(t,r),Q_{\pm m}(t,r)\big)^T$
and with
$\bar\Psi_{\pm m}(t,r)$
being the complex conjugate
(not a Hermitian conjugate).
One derives:
\begin{align}\label{re-dots}
&
2\begin{bmatrix}v\\-u\\0\\0
\end{bmatrix}
\Re\big(
\phi^*
\beta
\rho
\big)
=
\sum\sb\pm
\begin{bmatrix}W&\mp\frac m r W\\0&0\end{bmatrix}
\big(
\Psi_{\pm m}
\hh_{\ell,\pm m}
+
\bar\Psi_{\pm m}
\hh_{\ell,\mp m}
\big)
.
\end{align}

Let us prove Part~\itref{lemma-lm-2}.
One can see that for
$\phi$ from \eqref{sw}
and for $\rho\in\scrY_{\ell}$
one has
$\phi^*\beta\rho=0$
and hence one arrives at
$\calL\rho=\calL_0\rho$,
with $\calL_0$
from \eqref{def-l0}
being $\C$-linear and selfadjoint,
hence $\sigma\big(\calL\at{\scrY_{\ell}}\big)\subset\R$.
This completes the proof.
\end{proof}

We claim that the
union of the invariant subspaces
$\scrX_{\ell,m}$ 
and $\scrY_{\ell}$
contains the whole space
$L^2(\R^3,\C^4)$.

\begin{lemma}
\label{lemma-all-1}
For each
$\varrho\in\mathscr{S}(\R^3,\C^4)$,
there is a
\ac{(unique)???}
set of functions
$A_{\ell,m}(r),\,B_{\ell,m}(r)$,
$P_{\ell,m}(r),\,Q_{\ell,m}(r)$,
$R_{\ell}(r)$,
$S_{\ell}(r)\in\mathscr{S}(\R^3,\C^4)$,
with
$\ell\in\N_0$,
$m\in\Z$, $\abs{m}\le\ell$,
such that
\begin{align}\label{abr}
\sum_{\ell\in\N_0}
\sum\sb{-\ell\le m\le\ell}
\begin{bmatrix}
(A_{\ell,m}(r)+B_{\ell,m}(r)\sigma_r\varSigma)\hh_{\ell,m}\e_1
\nonumber
\\
\jj\sigma_r(P_{\ell,m}(r)+Q_{\ell,m}(r)\sigma_r\varSigma)\hh_{\ell,m}\e_1
\end{bmatrix}
+
\sum_{\ell\in\N_0}
\begin{bmatrix}
R_{\ell}(r)\hh_{\ell,-\ell}\e_2
\\
\jj\sigma_r S_{\ell}(r)\hh_{\ell,-\ell}\e_2
\end{bmatrix}
=\varrho(x).
\end{align}
\end{lemma}

The above lemma
follows by inspection.
\nb{Maybe a proof would be nice, anyway uniqueness seems true and the lemma can be formulated in  $L^2$.}

For the future convenience,
in the case $m\ne 0$,
we
decompose $\Re(\phi^*\beta\rho)$ into the two parts
corresponding to $\hh_{\ell,\pm m}$:
\[
\Re(\phi^*\beta\rho)=\sum_\pm\Re_\pm(\phi^*\beta\rho),
\]
\nb{$\Re_\pm$ is not a nice notation, could it be a linear map in $\rho$ notation?}
where
\begin{align}
\label{re-plus-minus}
\begin{array}{l}
\Re_{+}(\phi^*\beta\rho)
=
\begin{bmatrix}W&-\frac m r W\\0&0\end{bmatrix}
\Psi_{m}
\hh_{\ell,m}
+
\begin{bmatrix}W&\frac m r W\\0&0\end{bmatrix}
\bar\Psi_{-m}
\hh_{\ell,m}
,
\\[2ex]
\Re_{-}(\phi^*\beta\rho)
=
\begin{bmatrix}W&\frac m r W\\0&0\end{bmatrix}
\Psi_{-m}
\hh_{\ell,-m}
+
\begin{bmatrix}W&-\frac m r W\\0&0\end{bmatrix}
\bar\Psi_{m}
\hh_{\ell,-m},
\end{array}
\qquad
\ell\in\N,
\quad
1\le m\le \ell.
\end{align}
Using the explicit form \eqref{mis}
in the linearized equation \eqref{lin},
substituting
$\Delta \hh=-r^{-2}\varkappa_{\ell} \hh$,
and collecting
the coefficients at
$\hh_{\ell,\pm m}\e_1$ and $\sigma_r\varSigma \hh_{\ell,\pm m}\e_1$
(upper components)
as well as at
$\jj \sigma_r \hh_{\ell,\pm m}\e_1$
and $\jj\varSigma \hh_{\ell,\pm m}\e_1$
(lower components)
leads to the following system:
\begin{align}
\label{idt-apm}
\begin{cases}
\jj\p_t A_{\pm m}
=\Big(P_{\pm m}'+\frac{2}{r}P_{\pm m}
-\frac{\varkappa_{\ell}}{r^2}Q_{\pm m}\Big)
+(g-\omega) A_{\pm m}+2g'v\Re_\pm(\phi^*\beta\rho),
\\
\jj\p_t P_{\pm m}
=-\Big(A_{\pm m}'+\frac{\varkappa_{\ell}}{r^2}B_{\pm m}\Big)
-(g+\omega) P_{\pm m}-2g'u\Re_\pm(\phi^*\beta\rho),
\\
\jj\p_t B_{\pm m}
=(Q_{\pm m}'-P_{\pm m})+(g-\omega) B_{\pm m},
\\
\jj\p_t Q_{\pm m}
=-\Big(B_{\pm m}'+A_{\pm m}\Big)
-(g+\omega) Q_{\pm m},
\end{cases}
\end{align}
with
$\Re_\pm(\phi^*\beta\rho)$
given by
\eqref{re-plus-minus}.
We end up with the following:
\begin{align}\label{idt}
&
\jj\p_t
\begin{bmatrix}A_{\pm m}\\P_{\pm m}\\B_{\pm m}\\Q_{\pm m}
\end{bmatrix}
=
L_0
\begin{bmatrix}A_{\pm m}\\P_{\pm m}\\B_{\pm m}\\Q_{\pm m}
\end{bmatrix}
+
2g'\begin{bmatrix}v\\-u\\0\\0\end{bmatrix}\Re_\pm(\phi^*\beta\rho),
\end{align}
where
$L_0=L_0(\omega)$ is defined by
\begin{align}\label{def-l}
L_0(\omega)\begin{bmatrix}A\\P\\B\\Q
\end{bmatrix}
:=
\begin{bmatrix}
g-\omega&\p_r+\frac{2}{r}&0&-\frac{\varkappa_{\ell}}{r^2}
\\
-\p_r&-g-\omega&-\frac{\varkappa_{\ell}}{r^2}&0
\\
0&-1&g-\omega&\p_r
\\
-1&0&-\p_r&-g-\omega
\end{bmatrix}
\begin{bmatrix}A\\P\\B\\Q
\end{bmatrix}.
\end{align}


For perturbations corresponding to $\ell\in\N_0$, $m=0$
(this corresponds to the invariant subspace $\scrX_{\ell,0}$),
the linearized equation \eqref{idt}
takes the form
\begin{align}
\label{sys-l-0}
\p_t
\begin{bmatrix}\Psi_{\ell,0}\\\bar\Psi_{\ell,0}
\end{bmatrix}
=
-\jj
\left\{
\begin{bmatrix}\ L_0\ &\ 0\ \\[2ex]0&-L_0
\end{bmatrix}
+
\begin{bmatrix}W&0&W&0\\0&0&0&0
\\
-W&0&-W&0\\0&0&0&0
\end{bmatrix}
\right\}
\begin{bmatrix}\Psi_{\ell,0}\\\bar\Psi_{\ell,0}
\end{bmatrix},
\qquad
\ell\in\N_0.
\end{align}
Perturbations corresponding to spherical
harmonics of degree $\ell\in\N$
and nonzero orders $\pm m$
with $m\in\N$, $m\le\ell$,
are mixed
(cf. \eqref{wcw});
the linearized equation needs to contain two copies,
corresponding to $(\ell,m)$ and $(\ell,-m)$
(this corresponds to the invariant subspace
$\scrX_{\ell,m}$).
Then \eqref{idt} takes the following form:
\begin{align*}
\p_t\Psi_{m}
=-\jj L_0\Psi_{m}
-\jj\begin{bmatrix}W&-\frac{m}{r}W\\0&0\end{bmatrix}
\Psi_{m}
-\jj\begin{bmatrix}W&\frac{m}{r}W\\0&0\end{bmatrix}
\bar\Psi_{-m},
\\
\p_t\Psi_{-m}
=-\jj L_0\Psi_{-m}
-\jj\begin{bmatrix}W&\frac{m}{r}W\\0&0\end{bmatrix}
\Psi_{-m}
-\jj\begin{bmatrix}W&-\frac{m}{r}W\\0&0\end{bmatrix}
\bar\Psi_{m}.
\end{align*}
The above two equations
can be written as the following system:
\begin{align}
\label{linearization-lm}
\p_t
\begin{bmatrix}\Psi_{m}\\\bar\Psi_{-m}\end{bmatrix}
=
-\jj\left\{
\begin{bmatrix}\ L_0\ &\ 0\ \\[2ex]0&-L_0\end{bmatrix}
+
\begin{bmatrix}
W&-\frac{m}{r}W&W&\frac{m}{r}W
\\
0&0&0&0
\\
-W&\frac{m}{r}W&-W&-\frac{m}{r}W
\\
0&0&0&0
\end{bmatrix}
\right\}
\begin{bmatrix}\Psi_{m}\\\bar\Psi_{-m}\end{bmatrix},
\qquad
\ell,\,m\in\N,
\quad
m\le\ell.
\end{align}
Thus, the question of spectral stability
of a solitary wave reduces to studying the spectra of
operators
\[
\mathbf{A}_{\ell,m}
=
-\jj\left\{
\begin{bmatrix}\ L_0\ &\ 0\ \\[2ex]0&-L_0\end{bmatrix}
+
\begin{bmatrix}
W&-\frac{m}{r}W&W&\frac{m}{r}W
\\
0&0&0&0
\\
-W&\frac{m}{r}W&-W&-\frac{m}{r}W
\\
0&0&0&0
\end{bmatrix}
\right\},
\qquad
\ell,\,m\in\N_0,
\quad m\le\ell,
\]
with $L_0$
from \eqref{def-l}
(which depends on the degree $\ell$
via $\varkappa_{\ell}=\ell(\ell+1)$)
and with $W$ from \eqref{def-w}.

\nb{A comment would be welcome. Why this Lemma?}
The symmetry transformation~\ref{bo} beside inducing
bi-frequency solitary waves is responsible for the occurrence
of real eigenvalues for $\calL$.
\begin{lemma}
The value $\lambda=-2\omega$
is an eigenvalue of the operator $\calL$
of geometric multiplicity $2$,
with the corresponding eigenvectors
\[
\begin{bmatrix}
-\jj\sigma_r u\e_1\\v\e_1
\end{bmatrix}
\in\scrX_{1,-1},
\qquad
\begin{bmatrix}
-\jj\sigma_r u\e_2\\v\e_2
\end{bmatrix}
\in\scrX_{1,0}.
\]

\end{lemma}

\begin{proof}
Applying
$\jj\gamma^2\bmK$
to the relation
$\omega\phi=D_0\phi+g\beta\phi$,
one arrives at
$
\omega
\jj\gamma^2\bmK\phi
=
(-D_0-g\beta)\jj\gamma^2\bmK\phi,
$
which implies that
\begin{align}\label{def-psi-1}
\Psi_1
=\jj\gamma^2\bmK\phi
=
\begin{bmatrix}0&\jj\sigma_2\bmK\\-\jj\sigma_2\bmK&0
\end{bmatrix}
\begin{bmatrix}
v\e_1
\\
\jj\sigma_r u\e_1
\end{bmatrix}
=
\begin{bmatrix}
-\jj\sigma_r u\e_2
\\
v\e_2
\end{bmatrix}
=
\begin{bmatrix}
-\jj u
\begin{bmatrix}
e^{-\jj\phi}\sin\theta\\-\cos\theta
\end{bmatrix}
\\
v\begin{bmatrix}0\\1
\end{bmatrix}
\end{bmatrix},
\end{align}
is an eigenvector of $\calL_0=D_0+g\beta-\omega$
corresponding to eigenvalue $-2\omega$.
Since
$\Re\phi^*\beta\Psi_1=0$,
$\Psi_1$ is also an eigenvector of $\calL$
corresponding to the same eigenvalue $-2\omega$.
One also checks that
\begin{align}\label{def-psi-2}
\Psi_2
=
\begin{bmatrix}
-\jj\sigma_r u\e_1
\\
v\bm\e_1
\end{bmatrix}
=
\begin{bmatrix}
-\jj u
\begin{bmatrix}
\cos\theta\\e^{\jj\phi}\sin\theta
\end{bmatrix}
\\
v\begin{bmatrix}1\\0
\end{bmatrix}
\end{bmatrix}
\end{align}
satisfies
$\Re\phi^*\beta\Psi_2=0$,
so $\Psi_2$ is also an eigenvector of $\calL$
corresponding to eigenvalue $-2\omega$.
The inclusions
$\Psi_1\in\scrX_{1,-1}$
and
$\Psi_2\in\scrX_{1,0}$
can be deduced from the relations
\begin{align*}
&
(I_2+\sigma_r\varSigma_\Omega)
e^{-\jj\phi}\sin\theta\e_1
=
e^{-\jj\phi}\sin\theta\e_1
-\begin{bmatrix}
-\sin\theta
\\
e^{\jj\phi}\cos\theta
\end{bmatrix}e^{-\jj\phi}
+
\begin{bmatrix}
0\\
-e^{\jj\phi}\cos\theta
\end{bmatrix}e^{-\jj\phi}
=
\begin{bmatrix}
2e^{-\jj\phi}\sin\theta\\
-2\cos\theta
\end{bmatrix},
\\
&
(I_2+\sigma_r\varSigma_\Omega)\cos\theta\e_1
=
\begin{bmatrix}
\cos\theta
\\
e^{\jj\phi}\sin\theta
\end{bmatrix}.
\qedhere
\end{align*}
\end{proof}

\section{Spectral stability of bi-frequency solitary waves in 3D}
\label{sect-soler-bi}

A bi-frequency solitary wave
can be obtained
from a one-frequency solitary wave
$
e^{-\jj\omega t}
\begin{bmatrix}
v(r)\bm\xi_0
\\
\jj u(r)\sigma_r\bm\xi_0
\end{bmatrix}$,
$\bm\xi_0\in\C^2$,
$\abs{\bm\xi_0}=1$,
by the application of the $\mathbf{SU}(1,1)$ transformation:
given $a,\,b\in\C$,
$\abs{a}^2-\abs{b}^2=1$,
one has
\begin{equation}\label{bf}
\Big(a-b
\begin{bmatrix}0&\!\!\sigma_2\\-\sigma_2&\!\!0\end{bmatrix}
\bmK\Big)
e^{-\jj\omega t}
\begin{bmatrix}
v(r)\bm\xi_0
\\
\jj u(r)\sigma_r\bm\xi_0
\end{bmatrix}
=
e^{-\jj\omega t}
\begin{bmatrix}v(r)a\bm\xi_0
\\
\jj u(r)\sigma_r a\bm\xi_0
\end{bmatrix}
+
e^{\jj\omega t}
\begin{bmatrix}
-\jj u(r)\sigma_r b\sigma_2\bmK\bm\xi_0
\\
v(r)b\sigma_2\bmK\bm\xi_0
\end{bmatrix}.
\end{equation}
(We recall that $\sigma_2\bmK$ commutes with
$\jj\sigma_j$, $1\le j\le 3$.)
Moreover, by \cite{boussaid2018spectral},
it turns out that the expression
\begin{align}\label{bf-1}
\psi_{\bm\xi,\bm\eta,\omega}(t,x)
=
e^{-\jj\omega t}
\begin{bmatrix}
v(r)\bm\xi
\\
\jj u(r)\sigma_r\bm\xi
\end{bmatrix}
+
e^{\jj\omega t}
\begin{bmatrix}
-\jj u(r)\sigma_r\bm\eta
\\
v(r)\bm\eta
\end{bmatrix},
\end{align}
is an exact solution to \eqref{Soler}
with any
$\bm\xi,\,\bm\eta\in\C^2$,
$\abs{\bm\xi}^2-\abs{\bm\eta}^2=1$
in place of
$a\bm\xi_0$ and $b\sigma_2\bmK\bm\xi_0$
as in \eqref{bf}.
We note that
$\bm\xi_0$ and $\sigma_2\bmK\bm\xi_0$
are mutually orthogonal
since $\sigma_2$ is skew-symmetric;
it follows that
a bi-frequency solitary wave
\eqref{bf-1}
can not be represented in the form
\eqref{bf}
as long as $\langle\bm\xi,\bm\eta\rangle\ne 0$.

Assume that there is a solitary wave
solution $\psi_{\bm\xi,\bm\eta,\omega}$
with some $\bm\xi,\,\bm\eta\in\C^2$,
$\abs{\bm\xi}^2-\abs{\bm\eta}^2=1$,
$\omega\in(-m,m)$.
We claim that there is a decomposition
\nb{What is $\updelta$? something like $(\Re,\jj \Im)$?}
\[
\updelta:\;
L^2(\R^3,\C^4)\to\scrX\subset L^2(\R^3,\C^4)\times L^2(\R^3,\C^4),
\qquad
\Pi\scrX=L^2(\R^3,\C^4),
\qquad
\Pi\circ\updelta=I_{L^2},
\]
where
\nb{$\Pi$ for a sum? why not $S$ or anything similar?}
$\Pi:\,
L^2(\R^3,\C^4)\times L^2(\R^3,\C^4)
\to L^2(\R^3,\C^4)$, $(\rho_1,\rho_2)\mapsto\rho_1+\rho_2$,
and
an $\R$-linear, time-independent operator $\calA$,
\[
\calA=\calA(\bm\xi,\bm\eta,\omega)
:\;
L^2(\R^3,\C^4)\times L^2(\R^3,\C^4)
\to
L^2(\R^3,\C^4)\times L^2(\R^3,\C^4)
,
\]
with
$\dom(\calA)=H^1(\R^3,\C^4)\times H^1(\R^3,\C^4)$,
such that if
$(\rho_1(t,x)e^{-\jj\omega t},\rho_2(t,x)e^{\jj\omega t})
=\updelta\varrho\in\scrX$, with
\[
\varrho(t,x)=\rho_1(t,x)e^{-\jj\omega t}
+
\rho_2(t,x)e^{\jj\omega t},
\]
then
$\calA$
describes the linearized dynamics
of $\varrho$, 
in the sense that
$v
:=\begin{bmatrix}v_1\\v_2\end{bmatrix}
:=\calA\begin{bmatrix}\rho_1\\\rho_2\end{bmatrix}
$
satisfies
\begin{align*}
\jj\p_t\varrho
=
(\jj\p_t\rho_1+\omega\rho_1)e^{-\jj\omega t}
+
(\jj\p_t\rho_2-\omega\rho_2)e^{\jj\omega t}
=
(\jj v_1+\omega\rho_1)e^{-\jj\omega t}
+
(\jj v_2-\omega\rho_2)e^{\jj\omega t}
\\
=
D_0\varrho
+
g\big((\psi_{\bm\xi,\bm\eta,\omega}+\varrho)^*
\beta(\psi_{\bm\xi,\bm\eta,\omega}+\varrho)\big)
\beta(\psi_{\bm\xi,\bm\eta,\omega}+\varrho)
-g(\psi_{\bm\xi,\bm\eta,\omega}^*\beta\psi_{\bm\xi,\bm\eta,\omega})
\beta
\psi_{\bm\xi,\bm\eta,\omega}
+o(\varrho).
\end{align*}
We will call $\calA$ the operator
of linearization at
the bi-frequency solitary wave
$\psi_{\bm\xi,\bm\eta,\omega}$.
We will say that it is spectrally stable
if the corresponding $\C$-linear operator
$A:\;L^2(\R^3,\C^{16})\to L^2(\R^3,\C^{16})$
with $\dom(A)=H^1(\R^3,\C^{16})$,
defined by
\[
A\begin{bmatrix}\rho\\\varsigma\end{bmatrix}
=
\begin{bmatrix}
\calA(\frac{1}{2}(\rho+\bar\varsigma))
+
\jj\calA(\frac{1}{2\jj}(\rho-\bar\varsigma))
\\[0.5ex]
\overline{\calA(\frac{1}{2}(\rho+\bar\varsigma))}
+
\jj\overline{\calA(\frac{1}{2\jj}(\rho-\bar\varsigma))}
\end{bmatrix},
\]
has a purely imaginary spectrum
(above, for brevity, $\bar\varsigma=\bmK\varsigma$).
We note that the operator $A$
is $\C$-linear and agrees with the mapping
$\rho\mapsto\calA(\rho)$ in the sense that
$A\begin{bmatrix}\rho\\\bar\rho\end{bmatrix}
=
\begin{bmatrix}\calA(\rho)\\[0.5ex]\overline{\calA(\rho)}
\end{bmatrix}$.

\begin{theorem}\label{theorem-no-mixing}
Assume that $\bm\xi\in\C^2$,
$\abs{\bm\xi}\ge 1$, is parallel to $\e_1$,
and let $\bm\eta\in\C^2$ be such that
$\abs{\bm\xi}^2-\abs{\bm\eta}^2=1$.
Define the spaces
\nb{As defined $\frakX_{\ell}$ is made of elements of $
\cap L^2(\R^3,\C^4\times\C^4)$ which are dependent like $T=\begin{bmatrix}
                                                          0&-I\\I&0
                                                         \end{bmatrix}
$ (with some rotation to exchange $\e_1$ and $\e_2$ not forgetting $\xi\to\eta$) of each other. So $\Pi\rho=(\rho, T\rho)$? but $\updelta(\rho, T\rho)\neq \rho$? Note that $|\xi|=1$ makes $\eta=0$.}
\begin{align}
\label{both-para}
\frakX_{\ell}
&=
\mathop{\mathrm{Span}}\limits\sb{-\ell\le m\le\ell}
\left\{
\left(
\begin{bmatrix}
(A_{\ell,m}+B_{\ell,m}\sigma_r\varSigma)\hh_{\ell,m}\bm\xi
\\
\jj\sigma_r(P_{\ell,m}+Q_{\ell,m}\sigma_r\varSigma)\hh_{\ell,m}\bm\xi
\end{bmatrix}
,
\begin{bmatrix}
-\jj\sigma_r(\bar P_{\ell,m}+\bar Q_{\ell,m}\sigma_r\varSigma)
\hh_{\ell,-m}\bm\eta
\\
(\bar A_{\ell,m}+\bar B_{\ell,m}\sigma_r\varSigma)\hh_{\ell,-m}\bm\eta
\end{bmatrix}
\right)
\right\}
\nonumber
\\
&
\qquad\qquad\qquad\qquad\qquad
\qquad\qquad\qquad\qquad\qquad
\qquad\qquad\qquad
\cap L^2(\R^3,\C^4\times\C^4)
,
\end{align}
\begin{align}
\label{both-ortho}
\frakY_{\ell}
&=
\mathop{\mathrm{Span}}
\left\{
\left(
0,\,
\begin{bmatrix}
R_{\ell}\hh_{\ell,-\ell}\bm\xi^\perp
\\
\jj\sigma_r
S_{\ell}\hh_{\ell,-\ell}\bm\xi^\perp
\end{bmatrix}
\right)
\right\}
\cap L^2(\R^3,\C^4\times\C^4)
,
\end{align}
with
$\bm\xi^\perp=\abs{\bm\xi}\e_2$
and with
$A_{\ell,m}(r)$, $B_{\ell,m}(r)$, $P_{\ell,m}(r)$, $Q_{\ell,m}(r)$,
$R_{\ell}(r)$, $S_{\ell}(r)$
(where $\ell\in\N_0$ and $-\ell\le m\le\ell$)
complex-valued.
\nb{The spaces should be rewritten so as $A$, $B$, ... and $S$ are in $L^2$.}
\begin{enumerate}
\item
\label{theorem-no-mixing-1}
If $(\rho_1,\rho_2)\in\frakX_{\ell}$
or $(\rho_1,\rho_2)\in\frakY_{\ell}$,
then the quantity
$\psi_{\bm\xi,\bm\eta,\omega}^*\beta
(\rho_1 e^{-\jj\omega t}+\rho_2 e^{\jj\omega t})$
does not contain oscillatory factors
$e^{\pm 2\omega\jj t}$.
\nb{Then it is time independent}

More precisely, let
$\varrho(t,x)
=\rho_1(x)e^{-\jj\omega t}+\rho_2(x)e^{\jj\omega t}$,
with
$(\rho_1,\rho_2)$ valued in
$\frakX_{\ell}$:
\begin{align}
\label{def-rho-x}
\varrho(t,x)
&=
\sum_{m=-\ell}^\ell
\left(
e^{-\jj\omega t}
\begin{bmatrix}
(A_{\ell,m}+B_{\ell,m}\sigma_r\varSigma)\hh_{\ell,m}\bm\xi
\\
\jj\sigma_r(P_{\ell,m}+Q_{\ell,m}\sigma_r\varSigma)\hh_{\ell,m}\bm\xi
\end{bmatrix}
+
e^{\jj\omega t}
\begin{bmatrix}
-\jj\sigma_r(\bar P_{\ell,m}+\bar Q_{m\ell,}\sigma_r\varSigma)
\hh_{\ell,-m}\bm\eta
\\
(\bar A_{\ell,m}+\bar B_{\ell,m}\sigma_r\varSigma)\hh_{\ell,-m}\bm\eta
\end{bmatrix}
\right),
\end{align}
with $A_{\ell,m}$, $B_{\ell,m}$,
etc., functions of $r$
\nb{only and not $t$?}
.
Then
\[
\Re\big\{\psi_{\bm\xi,\bm\eta,\omega}^*\beta\varrho\big\}
=
\!
\sum_{m=-\ell}^\ell
\Re\big\{
(v A_{\ell,m} - u P_{\ell,m})\hh_{\ell,m}
+
(v B_{\ell,m}-u Q_{\ell,m})\big(
\bm\xi^*\sigma_r\varSigma \hh_{\ell,m}\bm\xi
+\bm\eta^*\sigma_r\varSigma \hh_{\ell,m}\bm\eta
\big)
\big\}.
\]

\item
\label{theorem-no-mixing-2}
For each $\ell\in\N_0$,
both $\frakX_{\ell}$
and
$\frakX_{\ell}\oplus\frakY_{\ell}$
are the invariant subspaces
of the linearization operator $\calA$.

\item
\label{theorem-no-mixing-3}
If
$X\in\frakX_\ell\oplus\frakY_\ell$
is an eigenvector of the
(complexification)
of the operator $\calA$ of linearization
at $\psi_{\bm\xi,\bm\eta,\omega}$
is such that its projection onto
$\frakY_\ell$ is nonzero,
then the corresponding eigenvalue
is purely imaginary.

\end{enumerate}
\end{theorem}

To include all harmonics of
particular degree $\ell\in\N_0$,
we will often
use below the following notations:
\begin{align}\label{AY}
A \hh=\sum_{m=-\ell}^\ell
A_{\ell,m}(t,r) \hh_{\ell,m}(\theta,\phi),
\qquad
B \hh=\sum_{m=-\ell}^\ell B_{\ell,m}(t,r) \hh_{\ell,m}(\theta,\phi),
\end{align}
and so on.

\begin{proof}
The fact that for
$\varrho(t,x)
=\rho_1(t,x)e^{-\jj\omega t}+\rho_2(t,x)e^{\jj\omega t}$,
with
$(\rho_1,\rho_2)$ valued in $\frakY_{\ell}$
(thus $\rho_1=0$),
the quantity $\psi_{\bm\xi,\bm\eta,\omega}^*\beta\varrho$
does not contribute oscillatory
factors $e^{\pm 2\jj\omega t}$
-- in fact, it vanishes identically --
is obtained by direct computation:
one just needs to notice that
$
\rho_2
=
\begin{bmatrix}
R_{\ell}\hh_{\ell,-\ell}\bm\xi^\perp
\\
\jj\sigma_r
S_{\ell}\hh_{\ell,-\ell}\bm\xi^\perp
\end{bmatrix}
$
is orthogonal to
$
\begin{bmatrix}
v\bm\xi
\\
\jj\sigma_r u\bm\xi
\end{bmatrix}
$.

Now let
$\varrho(t,x)
=\rho_1(t,x)e^{-\jj\omega t}+\rho_2(t,x)e^{\jj\omega t}$,
with
$(\rho_1,\rho_2)$ valued in $\frakX_{\ell}$.
Let us find the coefficient at
$e^{2\jj\omega t}$
in the expression for
$\psi_{\bm\xi,\bm\eta,\omega}^*\beta\varrho$,
coupling appropriate terms from
\eqref{bf-1} and \eqref{def-rho-x};
we claim that this coefficient,
given by
\begin{align}\label{igb}
\begin{bmatrix}
v\bm\xi
\\
\jj\sigma_r u\bm\xi
\end{bmatrix}^*
\beta
\begin{bmatrix}
-\jj
\sigma_r(\bar P_{\ell,m}+\bar Q_{\ell,m}\sigma_r\varSigma)\bar \hh_{\ell,m}
\bm\eta
\\
(\bar A_{\ell,m}+\bar B_{\ell,m}\sigma_r\varSigma)\bar \hh_{\ell,m})
\bm\eta
\end{bmatrix}
+
\begin{bmatrix}
(A_{\ell,m}+B_{\ell,m}\sigma_r\varSigma)\hh_{\ell,m}\bm\xi
\\
\jj\sigma_r(P_{\ell,m}+Q_{\ell,m}\sigma_r\varSigma)\hh_{\ell,m}\bm\xi
\end{bmatrix}^*
\beta
\begin{bmatrix}
-\jj
\sigma_r u\bm\eta
\\
v\bm\eta
\end{bmatrix},
\end{align}
vanishes.
Indeed,
since $\sigma_r$ is hermitian,
the coefficient at $\bar A_{\ell,m}$ is
$
-(\jj\sigma_r u\bm\xi)^*(\bar \hh_{\ell,m}\bm\eta)
+(\hh_{\ell,m}\bm\xi)^*(-\jj\sigma_r u\bm\eta)=0$.
The coefficient at $\bar B_{\ell,m}$
also vanishes
(due to \eqref{Eq:Important} and $\bar \hh_{\ell,m}=\hh_{\ell,-m}$):
\[
-(\jj\sigma_r u\bm\xi)^*(\sigma_r\varSigma\bar \hh_{\ell,m}\bm\eta)
+(\sigma_r\varSigma \hh_{\ell,m}\bm\xi)^*(-\jj\sigma_r u\bm\eta)
=
\jj(u\bm\xi)^*(\varSigma\bar \hh_{\ell,m}\bm\eta)
-\jj(\varSigma \hh_{\ell,m}\bm\xi)^*(u\bm\eta)=0.
\]
Similarly,
one checks that
the coefficients at $\bar P_{\ell,m}$ and $\bar Q_{\ell,m}$ also vanish,
\nb{Since it is pointwise identities $P$ and $Q$ are exchanged with $A$ and $B$ when $\xi$ and $\eta$ are replaced by with $\sigma_r\xi$ and $\sigma_r\eta$}
and hence \eqref{igb} equals zero.

Let us now write out
the terms in $\Re\{\psi_{\bm\xi,\bm\eta,\omega}^*\beta\varrho\}$
which do not contain $e^{\pm 2\jj\omega t}$:
\nb{I am not sure to follow well the spirit of the proof. Since the terms which possibly lead to $e^{\pm 2\jj\omega t}$ are already vanishing, what is the interest of the following? showing that it is time independent, which is clear as $e^{\pm 2\jj\omega t}$ terms vanish.}
\begin{align}
\label{stationary-terms}
&
\Re\left\{
\begin{bmatrix}
v\bm\xi
\nonumber
\\
\jj\sigma_r u\bm\xi
\end{bmatrix}^*
\beta
\begin{bmatrix}
(A+B\sigma_r\varSigma)\hh\bm\xi
\\
\jj\sigma_r(P+Q\sigma_r\varSigma)\hh\bm\xi
\end{bmatrix}
+
\begin{bmatrix}
-\jj
\sigma_r u\bm\eta
\\
v\bm\eta
\end{bmatrix}^*
\beta
\begin{bmatrix}
-\jj
\sigma_r(\bar P+\bar Q\sigma_r\varSigma)\bar \hh
\bm\eta
\\
(\bar A+\bar B\sigma_r\varSigma)\bar \hh
\bm\eta
\end{bmatrix}
\right\}
\\
&=
\Re\Big\{
\bm\xi^*\big(
v(A+B\sigma_r\varSigma)
-u(P+Q\sigma_r\varSigma)
\big)\hh
\bm\xi
-
\bm\eta^*\big(
v(\bar A+\bar B\sigma_r\varSigma)
-u(\bar P+\bar Q\sigma_r\varSigma)
\big)
\bar \hh\bm\eta
\Big\}
\nonumber
\\
&=
\Re\big\{
v A \hh- u P \hh
+
\bm\xi^*
(v B-u Q)\sigma_r\varSigma \hh
\bm\xi
-
\bm\eta^*
(v\bar B-u\bar Q)\sigma_r\varSigma
\bar \hh\bm\eta
\big\}
\nonumber
\\
&=
\Re\big\{
v A \hh- u P \hh
+
(v B-u Q)
\big(\bm\xi^*\sigma_r\varSigma \hh\bm\xi
+
\bm\eta^*\sigma_r\varSigma \hh\bm\eta
\big)
\big\}
.
\end{align}
In the last line, we used the relation
\[
\Re\big\{
\bm\eta^*
(v\bar B-u\bar Q)\sigma_r\varSigma
\bar \hh\bm\eta
\big\}
=
\Re\big\{
\bm\eta^*(v B-u Q)\sigma_j\sigma_r\p_j
\hh\bm\eta
\big\}
=
-\Re\big\{
\bm\eta^*(v B-u Q)\sigma_r\varSigma
\hh\bm\eta
\big\},
\]
with the second equality due to
$
\sigma_j\sigma_r\p_j
=-\sigma_r\sigma_j\p_j+2r^{-1}x\cdot\nabla
=-\sigma_r\sigma_j\p_j+2\p_r$
and
$\p_r \hh_{\ell,m}=0$.

Above,
in the products like $A \hh$,
the summation in $m$ is assumed;
cf. the notations \eqref{AY}.

This proves Part~\itref{theorem-no-mixing-1}.

\begin{remark}\label{remark-two-cases}
\nb{What is the goal of this remark? Better to put it outside the proof anyway.}
There are two ``endpoint'' cases:
when $\bm\eta$ is parallel to $\jj\sigma_2\bmK\bm\xi$
and
when $\bm\eta$ is parallel to $\bm\xi$.
In the first case,
one has
\[
\Re\big(
\bm\xi^*\sigma_r\varSigma \hh\bm\xi
+
\bm\eta^*\sigma_r\varSigma \hh\bm\eta
\big)
=
\Re\Big(
\abs{\bm\xi}^2
\frac{\bm\xi^*\sigma_r\varSigma \hh\bm\xi}{\abs{\bm\xi}^2}
-
\abs{\bm\eta}^2
\frac{\bm\xi^*\sigma_r\varSigma \hh\bm\xi}{\abs{\bm\xi}^2}
\Big)
=
\Re\Big(
\frac{\bm\xi^*}{\abs{\bm\xi}}
\sigma_r\varSigma \hh\frac{\bm\xi}{\abs{\bm\xi}}
\Big)
\]
(we took into account that
$
\bm\eta^*
\sigma_r\varSigma\hh\bm\eta
=
\frac{\abs{\bm\eta}^2}{\abs{\bm\xi}^2}
(\jj\sigma_2\bmK\bm\xi)^*\sigma_r\varSigma\hh\jj\sigma_2\bmK\bm\xi
=
\frac{\abs{\bm\eta}^2}{\abs{\bm\xi}^2}
\bm\xi^*\sigma_r\varSigma\bar\hh\bm\xi
=
-\frac{\abs{\bm\eta}^2}{\abs{\bm\xi}^2}
\bm\xi^*\sigma_r\varSigma\hh\bm\xi
$; here we assume that $\bm\xi$ is parallel to $\e_1$
and use \eqref{mssm-1})
and then the linearized operator
coincides with the linearization
at a one-frequency solitary wave
(corresponding to the spherical harmonic of degree $\ell$
and order $m$,
with the ``polarization''
given by $\bm\xi_0=\bm\xi/\abs{\bm\xi}\in\C^2$
in place of $\bm\xi$).
Indeed,
in this case the bi-frequency solitary wave
can be obtained from a one-frequency solitary wave
with the aid of an appropriate $\mathbf{SU}(1,1)$
transformation as in \eqref{bf}, hence
the one-frequency and bi-frequency solutions
share their stability properties.

If instead
$\bm\xi$ and $\bm\eta$
are parallel
(in this case, the bi-frequency solitary wave
cannot be obtained from a one-frequency solitary wave
with the aid of a transformation from $\mathbf{SU}(1,1)$), then
\[
\bm\xi^*\sigma_r\varSigma \hh\bm\xi
+
\bm\eta^*\sigma_r\varSigma \hh\bm\eta
=
\big(\abs{\bm\xi}^2+\abs{\bm\eta}^2\big)
\frac{\bm\xi^*\sigma_r\varSigma \hh\bm\xi}{\abs{\bm\xi}^2}
=
(1+2\abs{\bm\eta}^2)
\frac{\bm\xi^*\sigma_r\varSigma \hh\bm\xi}{\abs{\bm\xi}^2}.
\]
\end{remark}

Let us prove
Part~\itref{theorem-no-mixing-2}.
Since by Part~\itref{theorem-no-mixing-1}
the quantity
$\Re(\psi_{\bm\xi,\bm\eta,\omega}^*\beta\varrho)$
for $\varrho=\rho_1 e^{-\jj\omega t}+\rho_2 e^{\jj\omega t}$,
with $(\rho_1,\rho_2)$
valued in $\frakX_\ell\oplus\frakY_\ell$,
does not contain oscillating terms
$e^{\pm 2\jj\omega t}$,
we can split the first order system on
$A,\,B,\,P,\,Q,\,R,\,S$
which follows from
$\jj\p_t\psi=D_0\psi+g(\psi^*\beta\psi)\beta\psi$,
with
$\psi(t,x)=\psi_{\bm\xi,\bm\eta,\omega}(t,x)+\varrho(t,x)$,
\begin{align}
\psi(t,x)
=
e^{-\jj\omega t}
\begin{bmatrix}
(v+A \hh+B\sigma_r\varSigma \hh)\bm\xi
\\
\jj\sigma_r(u+P \hh+Q\sigma_r\varSigma \hh)\bm\xi
\end{bmatrix}
+
e^{\jj\omega t}
\Bigg\{
\begin{bmatrix}
-\jj
\sigma_r(u+\bar P\bar \hh+\bar Q\sigma_r\varSigma \bar \hh)
\bm\eta
\\
(v+\bar A\bar \hh+\bar B\sigma_r\varSigma \bar \hh)
\bm\eta
\end{bmatrix}
+
\begin{bmatrix}
R_\ell\hh_{\ell,-\ell}\bm\xi^\perp
\\
\jj\sigma_r S_\ell\hh_{\ell,-\ell}\bm\xi^\perp
\end{bmatrix}
\Bigg\},
\end{align}
into two groups of terms,
corresponding to coefficients at
$e^{-\jj\omega t}$ and $e^{\jj\omega t}$
(above,
in the products like $A \hh$,
the summation in $m$ is assumed;
cf. the notations \eqref{AY}).
The linearized equation breaks into
two following equations:
\nb{I do not understand the level of rigour here. How is the system split? Apart using analogies, I cannot follow. I am expecting to provide invariant subspaces and then project on them (especially if they are orthogonal).}
\begin{align}
\label{eqn1}
&
(\jj\p_t+\omega)
\begin{bmatrix}(A+B\sigma_r\varSigma)\hh\bm\xi
\\\jj\sigma_r(P+Q\sigma_r\varSigma)\hh\bm\xi
\end{bmatrix}
\\
\nonumber
&=
\begin{bmatrix}0&-\jj\varSigma\\-\jj\varSigma&0\end{bmatrix}
\begin{bmatrix}(A+B\sigma_r\varSigma)\hh\bm\xi
\\\jj\sigma_r(P+Q\sigma_r\varSigma)\hh\bm\xi
\end{bmatrix}
+
g\beta\begin{bmatrix}(A+B\sigma_r\varSigma)\hh\bm\xi
\\\jj\sigma_r(P+Q\sigma_r\varSigma)\hh\bm\xi\end{bmatrix}
+
2g'\Re(\psi_{\bm\xi,\bm\eta,\omega}^*\beta\varrho)
\beta
\begin{bmatrix}v\bm\xi
\\\jj\sigma_r u\bm\xi
\end{bmatrix},
\end{align}
\begin{align}
\label{eqn2-0}
&
(\jj\p_t-\omega)
\Bigg\{
\begin{bmatrix}-\jj\sigma_r(\bar P+Q\sigma_r\varSigma)\bar \hh\bm\eta
\\(\bar A+\bar B\sigma_r\varSigma)\bar \hh\bm\eta
\end{bmatrix}
+
\begin{bmatrix}
R_\ell\hh_{\ell,-\ell}\bm\xi^\perp
\\
\jj\sigma_r S_\ell\hh_{\ell,-\ell}\bm\xi^\perp
\end{bmatrix}
\Bigg\}
\\
\nonumber
&
=\begin{bmatrix}0&-\jj\varSigma\\-\jj\varSigma&0\end{bmatrix}
\Bigg\{
\begin{bmatrix}-\jj\sigma_r(\bar P+Q\sigma_r\varSigma)\bar \hh\bm\eta
\\(\bar A+\bar B\sigma_r\varSigma)\bar \hh\bm\eta
\end{bmatrix}
+
\begin{bmatrix}
R_\ell\hh_{\ell,-\ell}\bm\xi^\perp
\\
\jj\sigma_r S_\ell\hh_{\ell,-\ell}\bm\xi^\perp
\end{bmatrix}
\Bigg\}
\\
&
\qquad
+g
\beta
\Bigg\{
\begin{bmatrix}-\jj\sigma_r(\bar P+Q\sigma_r\varSigma)\bar \hh\bm\eta
\\(\bar A+\bar B\sigma_r\varSigma)\bar \hh\bm\eta
\end{bmatrix}
+
\begin{bmatrix}
R_\ell\hh_{\ell,-\ell}\bm\xi^\perp
\\
\jj\sigma_r S_\ell\hh_{\ell,-\ell}\bm\xi^\perp
\end{bmatrix}
\Bigg\}
+2g'\Re(\psi_{\bm\xi,\bm\eta,\omega}^*\beta\varrho)\beta
\begin{bmatrix}
-\jj\sigma_r u\bm\eta
\\ v\bm\eta
\end{bmatrix}.
\nonumber
\end{align}
Equation \eqref{eqn2-0}
follows from the two following equations
\nb{But as solution of \eqref{eqn2-0} does not necessarily come from solutions of both equations. The equivalence is not true without projection into invariant subspaces.}
(we are going to satisfy both):
\begin{align}
\label{eqn2}
&
(\jj\p_t-\omega)
\begin{bmatrix}-\jj\sigma_r(\bar P+Q\sigma_r\varSigma)\bar \hh\bm\eta
\\(\bar A+\bar B\sigma_r\varSigma)\bar \hh\bm\eta
\end{bmatrix}
\\
\nonumber
&
=\begin{bmatrix}0&-\jj\varSigma\\-\jj\varSigma&0\end{bmatrix}
\begin{bmatrix}-\jj\sigma_r(\bar P+Q\sigma_r\varSigma)\bar \hh\bm\eta
\\(\bar A+\bar B\sigma_r\varSigma)\bar \hh\bm\eta
\end{bmatrix}
+g
\beta
\begin{bmatrix}-\jj\sigma_r(\bar P+Q\sigma_r\varSigma)\bar \hh\bm\eta
\\(\bar A+\bar B\sigma_r\varSigma)\bar \hh\bm\eta
\end{bmatrix}
+2g'\Re(\psi_{\bm\xi,\bm\eta,\omega}^*\beta\varrho)\beta
\begin{bmatrix}
-\jj\sigma_r u\bm\eta
\\ v\bm\eta
\end{bmatrix},
\\
\label{eqn2-2}
&
(\jj\p_t-\omega)
\begin{bmatrix}
R_\ell\hh_{\ell,-\ell}\bm\xi^\perp
\\
\jj\sigma_r S_\ell\hh_{\ell,-\ell}\bm\xi^\perp
\end{bmatrix}
=
\begin{bmatrix}
0&-\jj\varSigma\\-\jj\varSigma&0
\end{bmatrix}
\begin{bmatrix}
R_\ell\hh_{\ell,-\ell}\bm\xi^\perp
\\
\jj\sigma_r S_\ell\hh_{\ell,-\ell}\bm\xi^\perp
\end{bmatrix}
+
g\beta
\begin{bmatrix}
R_\ell\hh_{\ell,-\ell}\bm\xi^\perp
\\
\jj\sigma_r S_\ell\hh_{\ell,-\ell}\bm\xi^\perp
\end{bmatrix}.
\end{align}
Both equations
\eqref{eqn1} and \eqref{eqn2}
will follow from the following system:
\begin{align}
\label{sys-0}
\begin{cases}
(\jj\p_t+\omega)(A+B\sigma_r\varSigma)\hh
=-\jj\varSigma\big(
\jj\sigma_r(P+Q\sigma_r\varSigma)\hh
\big)
+
g(A+B\sigma_r\varSigma)\hh
+2g'v\Re(\psi_{\bm\xi,\bm\eta,\omega}^*\beta\varrho),
\\
(\jj\p_t+\omega)
\jj\sigma_r(P+Q\sigma_r\varSigma)\hh
=
-\jj\varSigma(A+B\sigma_r\varSigma)\hh
-g\jj\sigma_r(P+Q\sigma_r\varSigma)\hh
-2\jj\sigma_rg'u\Re(\psi_{\bm\xi,\bm\eta,\omega}^*\beta\varrho).
\end{cases}
\end{align}
Indeed, to arrive at \eqref{eqn1},
one applies both equations
from \eqref{sys-0} to $\bm\xi$;
to arrive at \eqref{eqn2},
one multiplies both equations
by $\sigma_2\bmK$ on the left
(which anticommutes with each of
$\jj$, $\sigma_i$, $\sigma_r$, and $\varSigma$),
by $(\sigma_2\bmK)^{-1}=-\sigma_2\bmK$ on the right,
and then applies the resulting relations to $\bm\eta$.
Taking into account that,
by Lemma~\ref{lemma-sigma},
\begin{align}
\label{tia-1}
\varSigma\circ\sigma_r=
2\p_r+2r^{-1}-\sigma_r\varSigma
=
\p_r+2r^{-1}-r^{-1}\sigma_r\varSigma_\Omega,
\end{align}
we rewrite the first term
in the right-hand side of the first equation in
\eqref{sys-0}
as
\begin{align}
\label{oh}
&
\varSigma\sigma_r(P\hh+Q\sigma_r\varSigma\hh)
=
P'\hh+2r^{-1}P\hh+Q\Delta\hh
+
(-P+Q')
\sigma_r\varSigma
\end{align}
(with summation in $m$ as in \eqref{AY}),
where
for harmonics of degree $\ell\in\N_0$ we have
$\Delta \hh=-r^{-2}\varkappa_{\ell} \hh$;
this allows to rewrite
the first equation from \eqref{sys-0} as
\begin{align}
\label{sys-1}
&(\jj\p_t+\omega)(A+B\sigma_r\varSigma)\hh
\nonumber
\\
&\qquad\qquad
=
P' \hh+2r^{-1}P \hh-\frac{\varkappa_{\ell}}{r^2}Q \hh
+
(-P+Q')\sigma_r\varSigma \hh
+
g(A+B\sigma_r\varSigma)\hh
+2g'v\Re(\psi_{\bm\xi,\bm\eta,\omega}^*\beta\varrho).
\end{align}
Now the evolution equation on $A$
can be obtained by
collecting the terms with $\hh_{\ell,m}$ only
\nb{Strange, we need to project?}
(that is, dropping the terms with $\sigma_r\varSigma \hh_{\ell,m}$)
from \eqref{sys-1}:
\begin{align}\label{ata}
(\jj\p_t+\omega)A \hh
=
P' \hh+2r^{-1}P \hh+Q\Delta \hh
+A \hh g
+2 g' v\Re(\psi_{\bm\xi,\bm\eta,\omega}^*\beta\varrho).
\end{align}
Evolution equation on $B$
is obtained by
collecting the coefficients at $\sigma_r\varSigma \hh_{\ell,m}$
in the first equation in \eqref{sys-0}:
\nb{This part is even stranger}
\begin{align}\label{atb}
(\jj\p_t+\omega)B_{\ell,m}
=
(-P_{\ell,m}+Q_{\ell,m}')
+
B_{\ell,m} g
.
\end{align}
Now we need the following identity:
\[
\varSigma(A\hh+B\sigma_r\varSigma\hh)
=
A'\sigma_r \hh+A\varSigma \hh
+B'\varSigma \hh+B\varSigma\sigma_r\varSigma \hh
=
\Big(A'\sigma_r+A\varSigma
+B'\varSigma
+B
\sigma_r\frac{\varkappa_{\ell}}{r^2}
\Big)\hh.
\]
Above, we took into account that, by
Lemma~\ref{lemma-identities},
$
\varSigma\sigma_r\varSigma \hh
=
(3-n)r^{-1}\varSigma \hh-\sigma_r\Delta \hh$
(where we take $n=3$)
and that
$\Delta \hh=-r^{-2}\varkappa_{\ell} \hh$.
This allows to rewrite
the second equation from \eqref{sys-0}
as
\begin{align}\label{sys-2}
&(\jj\p_t+\omega)
\jj\sigma_r(P+Q\sigma_r\varSigma)\hh
\\
\nonumber
&
\qquad
=
-\jj
\Big(A'\sigma_r+A\varSigma
+B'\varSigma
+
B
\sigma_r\frac{\varkappa_{\ell}}{r^2}
\Big)\hh
-g\jj\sigma_r(P+Q\sigma_r\varSigma)\hh
-2\jj\sigma_rg'u\Re(\psi_{\bm\xi,\bm\eta,\omega}^*\beta\varrho).
\end{align}
Evolution equation on $P$
is obtained by
collecting the terms with $\sigma_r \hh$
from equation \eqref{sys-2}:
\begin{align}\label{atp}
(\jj\p_t+\omega)\jj\sigma_r P \hh
=
-\jj\Big(A'\sigma_r \hh
+B\sigma_r\frac{\varkappa_{\ell}}{r^2}\hh
\Big)
-\jj\sigma_r g P \hh
-2\jj g' u\sigma_r\Re(\psi_{\bm\xi,\bm\eta,\omega}^*\beta\varrho).
\end{align}
Collecting the coefficients
at terms with $\varSigma \hh_{\ell,m}$ in \eqref{sys-2},
we arrive at the evolution equation for $Q_{\ell,m}$:
\begin{align}\label{atq}
(\jj\p_t+\omega)\jj Q_{\ell,m}
=
-\jj\big(A_{\ell,m}
+B_{\ell,m}'
\big)
-\jj g Q_{\ell,m}.
\end{align}

Equation \eqref{eqn2-2}
leads to the following system:
\begin{align}
\label{tfs}
\begin{cases}
(\jj\p_t-\omega)
R_\ell\hh_{\ell,-\ell}\bm\xi^\perp
=
\big(2\p_r+2r^{-1}-\sigma_r\varSigma\big)
S_\ell\hh_{\ell,-\ell}\bm\xi^\perp
+g R_\ell\hh_{\ell,-\ell}\bm\xi^\perp,
\\
(\jj\p_t-\omega)
S_\ell\hh_{\ell,-\ell}\bm\xi^\perp
=
-\sigma_r\varSigma(R_\ell\hh_{\ell,-\ell})\bm\xi^\perp
-g S_\ell\hh_{\ell,-\ell}\bm\xi^\perp.
\end{cases}
\end{align}
Above, we multiplied the second equation
by $-\jj\sigma_r$.

Using~\eqref{tia-1}
and noting that
$\sigma_r\varSigma\hh_{\ell,-\ell}\e_2
=
-r^{-1}\ell\hh_{\ell,-\ell}\e_2$
(cf. Lemma~\ref{lemma-mssm}),
we rewrite \eqref{tfs} as
\begin{align*}
\begin{cases}
(\jj\p_t-\omega)
R_\ell\hh_{\ell,-\ell}\bm\xi^\perp
=
\big(\p_r+2 r^{-1}+\ell r^{-1}\big)
S_\ell\hh_{\ell,-\ell}\bm\xi^\perp
+g R_\ell\hh_{\ell,-\ell}\bm\xi^\perp,
\\
(\jj\p_t-\omega)
S_\ell\hh_{\ell,-\ell}\bm\xi^\perp
=
-(\p_r-\ell r^{-1})R_\ell\hh_{\ell,-\ell}\bm\xi^\perp
-g S_\ell\hh_{\ell,-\ell}\bm\xi^\perp,
\end{cases}
\end{align*}
which will follow from
\begin{align}
\label{atrs}
\begin{cases}
\jj\p_t R_\ell
=
\big(\p_r+2 r^{-1}+\ell r^{-1}\big)S_\ell
+g R_\ell+\omega R_\ell,
\\
\jj\p_t S_\ell
=
-(\p_r-\ell r^{-1}) R_\ell-g S_\ell+\omega S_\ell.
\end{cases}
\end{align}
Collecting equations
\eqref{ata},
\eqref{atb},
\eqref{atp}
(where we drop the common factor $\jj\sigma_r$),
\eqref{atq},
and \eqref{atrs},
we arrive at the following system,
which describes the linearization
at a bi-frequency solitary wave $\psi_{\bm\xi,\bm\eta,\omega}$:
\begin{align}\label{sys-bf-0}
\begin{cases}
\jj\p_t A \hh
=
\big(P'+\frac{2}{r}P-\frac{\varkappa_{\ell}}{r^2}Q
+(g-\omega)A\big)\hh
+2 g' v\Re(\psi_{\bm\xi,\bm\eta,\omega}^*\beta\varrho),
\\
\jj\p_t P \hh
=
-(A'+\frac{\varkappa_{\ell}}{r^2}B+(g-\omega)P)\hh
-2 g' u\Re(\psi_{\bm\xi,\bm\eta,\omega}^*\beta\varrho),
\\
\jj\p_t B_{\ell,m}
=
(-P_{\ell,m}+Q_{\ell,m}')
+
(g-\omega)B_{\ell,m},
\\
\jj\p_t Q_{\ell,m}
=
-A_{\ell,m}-B_{\ell,m}'-\frac{3-n}{r}B_{\ell,m}-(g+\omega)Q_{\ell,m},
\\
\jj\p_t R_\ell
=
\big(\p_r+2+\ell r^{-1}\big)S_\ell
+g R_\ell+\omega R_\ell,
\\
\jj\p_t S_\ell
=
-(\p_r-\ell r^{-1}) R_\ell-g S_\ell+\omega S_\ell.
\end{cases}
\end{align}
We note that
the first two equations
of \eqref{sys-bf-0}
contain summation in $m$,
so that
$A \hh=\sum_{-\ell\le m\le\ell}A_{\ell,m} \hh_{\ell,m}$ and so on
(cf. \eqref{AY});
these equations could be projected onto
the spherical harmonics $\hh_{\ell,m}$
with particular degree $\ell$ but different
order $m$ satisfying $\abs{m}\le\ell$.
We note that
by Lemma~\ref{lemma-same-degree},
the term with $\Re(\psi_{\bm\xi,\bm\eta,\omega}^*\beta\varrho)$
contains linear combination of spherical harmonics
of degree $\ell$,
thus providing the interaction of harmonics
of the same $\ell$ but different $m$.

The system
\eqref{sys-bf-0}
can be written as
$\p_t \Psi=\calA\Psi$,
with $\calA$
the
operator of linearization
at the bi-frequency soltiary wave
$\psi_{\bm\xi,\bm\eta,\omega}$;
this operator is non-$\C$-linear because of $\Re(...)$-term.
This system shows that
$\frakX_\ell\oplus\frakY_\ell$
is the invariant subspace
of the linearization operator,
\nb{It is not clear that \eqref{sys-bf-0} is equivalent to the linearized evolution. Solutions of \eqref{sys-bf-0}
give solutions of the linearized system. My opinion is that the proof should be written this way instead of deduced, which is not the case anyway, from the linearized system. The completion of the equivalence should use uniqueness of Cauchy problem.
}
which we can formulate as follows:
if $\varrho=\rho_1 e^{-\jj\omega t}+\rho_2 e^{\jj\omega t}$
with $(\rho_1,\rho_2)\in\frakX_\ell\oplus\frakY_\ell$,
then
$(\p_t\rho_1-\jj\omega\rho_1,\p_t\rho_2+\jj\omega\rho_2)
\in\frakX_\ell\oplus\frakY_\ell$.
One can also see that the last two equations
for $R_\ell$ and $S_\ell$
in \eqref{sys-bf-0}
are uncoupled from the first four,
so if $R_\ell$ and $S_\ell$ vanish at some moment of time,
then $\p_t R_\ell$ and $\p_t S_\ell$ also vanish.
(That is, if $\varrho=\rho_1 e^{-\jj\omega t}+\rho_2 e^{\jj\omega t}$
with $(\rho_1,\rho_2)\in\frakX_\ell$,
then
$(\p_t\rho_1-\jj\omega\rho_1,\p_t\rho_2+\jj\omega\rho_2)
\in\frakX_\ell$.)
This completes the proof
of Part~\itref{theorem-no-mixing-2}.

The last two equations from
\eqref{sys-bf-0}
can be written as
\[
\jj\p_t\begin{bmatrix}R_\ell\\S_\ell\end{bmatrix}
=\calA_{RS}
\begin{bmatrix}R_\ell\\S_\ell\end{bmatrix},
\qquad
\calA_{RS}:=\begin{bmatrix}
g+\omega&\p_r+2r^{-1}+\ell r^{-1}
\\
-\p_r+\ell r^{-1}&-g+\omega
\end{bmatrix}.
\]
Since the operator $\calA_{RS}$ is symmetric,
it follows that
eigenvalues of $\calA$
such that the corresponding
eigenvectors
have nonzero components $R_\ell$ and $S_\ell$
are purely imaginary.
This proves
Part~\itref{theorem-no-mixing-3}.
\end{proof}

By Theorem~\ref{theorem-no-mixing}~\itref{theorem-no-mixing-3},
to study the spectral stability
(absence of eigenvalues with nonzero real part)
of the linearization operator
$\calA$
corresponding to the system
\eqref{sys-bf-0},
we may assume that
$R_\ell$ and $S_\ell$ are equal to zero,
which we do from now on,
considering perturbations valued in $\frakX_\ell$.
In the matrix form, \eqref{sys-bf-0}
with $R_\ell=S_\ell=0$
reads as follows:
\begin{align}
\label{sys-bf}
\jj\p_t
\begin{bmatrix}A\\P\\B\\Q
\end{bmatrix}
=
\begin{bmatrix}
g-\omega&\p_r+\frac{2}{r}&0&-\frac{\varkappa_{\ell}}{r^2}
\\
-\p_r&-g-\omega&-\frac{\varkappa_{\ell}}{r^2}&0
\\
0&-1&g-\omega&\p_r
\\
-1&0&-\p_r
&-g-\omega
\end{bmatrix}
\begin{bmatrix}A\\P\\B\\Q
\end{bmatrix}
+
2g'
\begin{bmatrix}v\\-u\\0\\0
\end{bmatrix}
\Re(\psi_{\bm\xi,\bm\eta,\omega}^*\beta\varrho)
,
\end{align}
where,
by Theorem~\ref{theorem-no-mixing}~\itref{theorem-no-mixing-1},
the quantity
$2\Re(\psi_{\bm\xi,\bm\eta,\omega}^*\beta\varrho)$
is given by
\begin{align}
\label{where}
2\Re(\psi_{\bm\xi,\bm\eta,\omega}^*\beta\varrho)
&=
2\Re\big\{
v A \hh- u P \hh
+
(v B-u Q)
\big[\bm\xi^*\sigma_r\varSigma \hh\bm\xi
+\bm\eta^*\sigma_r\varSigma \hh\bm\eta\big]
\big\}
\\
\nonumber
&=
v A \hh- u P \hh
+
v \bar A \bar \hh- u \bar P \bar \hh
\\
\nonumber
&
\qquad
+
 (v B-u Q)
 \big[\bm\xi^*\sigma_r\varSigma \hh\bm\xi
 +\bm\eta^*\sigma_r\varSigma \hh\bm\eta\big]
+
 (v\bar B-u\bar Q)
 \big[\bm\xi^*\varSigma\sigma_r\bar \hh\bm\xi
 +\bm\eta^*\varSigma\sigma_r\bar \hh\bm\eta\big].
\end{align}
By Lemma~\ref{lemma-same-degree},
there are $C_{m,k}(\bm\xi)\in\C$
such that
$
\bm\xi^*\sigma_r\varSigma_\Omega \hh_{\ell,k}\bm\xi
=
\sum\sb{k=-\ell}^{\ell}
\hh_{\ell,m}C_{m,k}(\bm\xi)$
for all $-\ell\le k\le\ell$.
Define
\[
C_{m,k}
=C_{m,k}(\bm\xi)+C_{m,k}(\bm\eta),
\qquad
-\ell\le k,\,m\le\ell.
\]
(In the case when $\bm\xi$ and $\bm\eta$
are parallel to $\e_1$,
Lemma~\ref{lemma-mssm} yields
$C_{m,k}=-(\abs{\bm\xi}^2+\abs{\bm\eta}^2)m\delta_{m,k}$,
with $-\ell\le k,\,m\le\ell$).
By Theorem~\ref{theorem-no-mixing}~\itref{theorem-no-mixing-1},
the expression \eqref{where} takes the form
\begin{align*}
2\Re(\psi_{\bm\xi,\bm\eta,\omega}^*\beta\varrho)
&=
\sum_{m=-\ell}^\ell
\Big\{
(v A_{\ell,m}-u P_{\ell,m})\hh_{\ell,m}
+(v\bar A_{\ell,m}-u\bar P_{\ell,m})\hh_{\ell,-m}
\phantom{\sum_{\ell}^\ell}
\\
&
\qquad\qquad
+
\sum\sb{k=-\ell}^\ell
r^{-1}
\big(
\hh_{\ell,m}C_{m,k}(v B_{\ell,k}-u Q_{\ell,k})
+\hh_{\ell,-m}\bar C_{m,k}(v\bar B_{\ell,k}-u\bar Q_{\ell,k})
\big)
\Big\}.
\end{align*}
Collecting
in \eqref{sys-bf}
terms with harmonics $\hh_{\ell,m}$,
the linearized dynamics
can be written as the following system of equations
on $\Psi_{\ell,m}$,
with $-\ell\le m\le\ell$:
\begin{align*}
\jj\p_t
\Psi_{\ell,m}
=
L_0\Psi_{\ell,m}
+
\begin{bmatrix}W&0
\\
0&0
\end{bmatrix}
\Psi_{\ell,m}
+
\begin{bmatrix}W&0
\\
0&0
\end{bmatrix}
\bar\Psi_{\ell,-m}
+
\frac{1}{r}
\begin{bmatrix}0&C_{m,k}W
\\
0&0
\end{bmatrix}
\Psi_{\ell,k}
+
\frac{1}{r}
\begin{bmatrix}0&\bar C_{-m,k} W
\\
0&0
\end{bmatrix}
\bar\Psi_{\ell,k}
.
\end{align*}
Above, the summation in $k$
from $-\ell$ to $\ell$ is assumed.
We rewrite the above equations
in terms of $\C$-linear operators:
\begin{align*}
\p_t
\begin{bmatrix}\Psi_{\ell,m}\\\bar\Psi_{\ell,m}
\end{bmatrix}
=-\jj
\left\{
\Biggr(
\begin{bmatrix}\ L_0\ &0\\[2ex]0&-L_0
\end{bmatrix}
+
\begin{bmatrix}W&0&0&0
\\
0&0&0&0
\\
0&0&-W&0
\\
0&0&0&0
\end{bmatrix}
\Biggr)
\begin{bmatrix}\Psi_{\ell,m}\\\bar\Psi_{\ell,m}\end{bmatrix}
+
\begin{bmatrix}
0&0&W&0
\\
0&0&0&0
\\
-W&0&0&0
\\
0&0&0&0
\end{bmatrix}
\begin{bmatrix}\Psi_{\ell,-m}\\\bar\Psi_{\ell,-m}\end{bmatrix}
\right.
\\[2ex]
\qquad
\left.
+
\sum_{k=-\ell}^{\ell}\frac{1}{r}
\begin{bmatrix}0&C_{m,k} W&0&\bar C_{-m,k} W
\\
0&0&0&0
\\
0&-\bar C_{m,k} W&0&-C_{-m,k} W
\\
0&0&0&0
\end{bmatrix}
\right\}
\begin{bmatrix}\Psi_{\ell,k}\\\bar\Psi_{\ell,k}\end{bmatrix}
,
\qquad
-\ell\le m\le\ell.
\end{align*}
The above equations
can be written as a system,
\[
\p_t\chi
=-\jj\left\{
I_{2\ell+1}\otimes
\left(
\begin{bmatrix}\ L_0\ &\ 0\ \\[2ex]0&-L_0
\end{bmatrix}
+
\begin{bmatrix}W&\frac{1}{r}C\otimes W&W&\frac{1}{r}\bar D\otimes W
\\
0&0&0&0
\\
-W&-\frac{1}{r}\bar C\otimes W&-W&-\frac{1}{r}D\otimes W
\\
0&0&0&0
\end{bmatrix}
\right)
\right\}
\chi,
\]
where
$I_{2\ell+1}$ refers to all orders of the spherical harmonics
of degree $\ell$
and $D_{m,k}=C_{-m,k}$
and where
$\chi=\big(
\Psi_{\ell,-\ell},\bar\Psi_{\ell,-\ell},
\dots,
\Psi_{\ell,m},\bar\Psi_{\ell,m},
\dots,
\Psi_{\ell,\ell},\bar\Psi_{\ell,\ell}
\big)^T$.

\nb{I guess the following is related to the comment "(In the case when $\bm\xi$ and $\bm\eta$
are parallel to $\e_1$,
Lemma~\ref{lemma-mssm} yields
$C_{m,k}=-(\abs{\bm\xi}^2+\abs{\bm\eta}^2)m\delta_{m,k}$,
with $-\ell\le k,\,m\le\ell$).
'' made earlier?}
Considering the case
$\ell\in\N_0$, $-\ell\le m\le\ell$,
$\bm\eta=\nu\e_1$ with $\nu\ge 0$,
$\bm\xi=\sqrt{1+\nu^2}\e_1$,
when $C_{m,k}=-(\abs{\bm\xi}^2+\abs{\bm\eta}^2)m
\delta_{m,k}$,
we obtain the following representation
\nb{Again, to be a representation, it should be equivalent. It seems to me that the writing should be
solutions to \eqref{linearization-lm-bi} provide solutions to the linearized system. Due to uniqueness
in Cauchy problem and uniqueness of the decomposition in Lemma~\ref{lemma-all-1} (or is it Section~\ref{sect-init}?), these are all the solutions.
}
for the linearization operator
in the invariant subspace
corresponding to the spherical harmonics
of degree $\ell$ and orders $\pm m$:
\begin{align}
\label{linearization-lm-bi}
\p_t
\begin{bmatrix}
\Psi_{\ell,m}\\\bar\Psi_{\ell,-m}
\end{bmatrix}
=-\jj
\left\{
\begin{bmatrix}
L_0&0
\\[2ex]
0&-L_0
\end{bmatrix}
+
\begin{bmatrix}
W&-(1+2\nu^2)\frac{m}{r}W&W&(1+2\nu^2)\frac{m}{r}W
\\
0&0&0&0
\\
-W&(1+2\nu^2)\frac{m}{r}W&-W&-(1+2\nu^2)\frac{m}{r}W
\\
0&0&0&0
\end{bmatrix}
\right\}
\begin{bmatrix}
\Psi_{\ell,m}\\\bar\Psi_{\ell,-m}
\end{bmatrix}.
\end{align}

\begin{remark}
\label{remark-better}
Comparing the above expression with \eqref{linearization-lm},
we see that
for particular value of $\nu>1$ and $m\in\Z$
the above operator
corresponds to the operator
in \eqref{linearization-lm}
with the value of magnetic quantum number
$m'\sim(1+2\nu^2)m$.
If
\nb{The ''if'' here is suprising. Why not starting with "A situation was observed for solitary waves
in (2+1)D Soler model \cite{PhysRevLett.116.214101} is as follows"  and in the end concluding that ''In this sense,
bi-frequency solitary waves
may have \emph{better} stability
than one-frequency solitary waves.''? }
at a particular value
of $\omega$
the operator
in \eqref{linearization-lm}
is spectrally stable for $m=0$, unstable for smaller values
of $m\ne 0$
and becomes spectrally stable again for $\abs{m}\ge m_0$
with some $m_0\in\N$
(such a situation was observed for solitary waves
in (2+1)D Soler model \cite{PhysRevLett.116.214101}),
then the linearization at a bi-frequency solitary wave
will be stable if $\nu$ is large enough, so that
$(1+2\nu^2)\ge m_0$.
In this sense,
bi-frequency solitary waves
may have \emph{better} stability
than one-frequency solitary waves.
\end{remark}

\section{Decomposing the initial data
into $\sum_{\ell,m}\frakX_{\ell,m}+\sum_{\ell}\frakY_{\ell}$}
\nb{Is it linked with Lemma~\ref{lemma-all-1}?}
\label{sect-init}

We need to show
that any initial perturbation
of a two-frequency solitary wave
can be decomposed
into pieces which belong to
invariant subspaces
$\frakX_{\ell,m}$
and $\frakY_{\ell}$
introduced in \eqref{both-para} and \eqref{both-ortho}.
We will only do this in the case
when $\bm\xi$ and $\bm\eta$ are parallel,
which is the most interesting one
according to Remark~\ref{remark-two-cases}.

\begin{lemma}
\label{lemma-all-2}
For
$\bm\xi,\,\bm\eta\in\C^2$
both parallel to $\e_1$
and satisfying
$\abs{\bm\xi}^2-\abs{\bm\eta}^2=1$
and for each
$\varrho\in\mathscr{S}(\R^3,\C^4)$
there is a set of functions
$A_{\ell,m}(r),\,B_{\ell,m}(r)$,
$P_{\ell,m}(r),\,Q_{\ell,m}(r)$,
$R_{\ell}(r)$,
$S_{\ell}(r)\in\mathscr{S}(\R^3,\C^4)$,
with
$\ell\in\N_0$,
$m\in\Z$, $\abs{m}\le\ell$,
such that
\begin{align}\label{abr}
\sum_{\ell\in\N_0}
\sum\sb{-\ell\le m\le\ell}
\left(
\begin{bmatrix}
(A_{\ell,m}(r)+B_{\ell,m}(r)\sigma_r\varSigma)\hh_{\ell,m}\bm\xi
\nonumber
\\
\jj\sigma_r(P_{\ell,m}(r)+Q_{\ell,m}(r)\sigma_r\varSigma)\hh_{\ell,m}\bm\xi
\end{bmatrix}
+
\begin{bmatrix}
-\jj
\sigma_r(\bar P_{\ell,m}(r)+\bar Q_{\ell,m}(r)\sigma_r\varSigma)
\hh_{\ell,-m}
\bm\eta
\\
(\bar A_{\ell,m}(r)+\bar B_{\ell,m}(r)\sigma_r\varSigma)\hh_{\ell,-m}
\bm\eta
\end{bmatrix}
\right)
\\
+
\sum_{\ell\in\N_0}
\begin{bmatrix}
R_{\ell}(r)\hh_{\ell,-\ell}\bm\xi^\perp
\\
\jj\sigma_r S_{\ell}(r)
\hh_{\ell,-\ell}\bm\xi^\perp
\end{bmatrix}
=\varrho(x).
\end{align}
\nb{As in Lemma~\ref{lemma-all-1} it should be formulated in  $L^2$.}
\end{lemma}

\begin{proof}
We note that it is enough to consider the case
\begin{align}\label{one-two}
\varrho(x)
=
\begin{bmatrix}
\varrho_1(x)
\\
\varrho_2(x)
\end{bmatrix}\in\C^4,
\qquad
\mbox{where $\varrho_2(x)\equiv 0$;}
\end{align}
the case
with $\varrho_1\equiv 0$
is reduced to \eqref{one-two}
by applying
$\begin{bmatrix}0&\sigma_2\\\sigma_2&0\end{bmatrix}$
to both sides of \eqref{abr}.
Also, for our convenience,
dividing \eqref{abr}
by $\abs{\bm\xi}$,
we may assume that
\begin{align}\label{such}
\bm\xi=\e_1,
\qquad
\bm\xi\sp\perp=\e_2,
\qquad
\bm\eta=\epsilon\e_1,
\qquad
0\le\epsilon<1.
\end{align}
We need to solve (we do not indicate explicitly
the dependence on $r$
and the summation in $\ell$ and $m$):
\begin{align}\label{c1}
(A_{\ell,m}+B_{\ell,m}\sigma_r\varSigma)\hh_{\ell,m}\e_1
-
\jj\epsilon\sigma_r
(\bar P_{\ell,m}+\bar Q_{\ell,m}\sigma_r\varSigma)\hh_{\ell,-m}\e_2
+R_{\ell}\hh_{\ell,-\ell}\e_2
=\varrho_1,
\\
\label{c2}
\jj\sigma_r
(P_{\ell,m}+Q_{\ell,m}\sigma_r\varSigma)\hh_{\ell,m}\e_1
+
\epsilon(\bar A_{\ell,m}+\bar B_{\ell,m}\sigma_r\varSigma)
\hh_{\ell,-m}\e_2
+\jj\sigma_r
S_{\ell}\hh_{\ell,-\ell}\e_2=0.
\end{align}
Acting onto the second relation by
the operator $-\jj\sigma_2\bmK$
and taking into account that
this operator commutes with
both $\jj\sigma_r$ and $\sigma_r\varSigma$,
and that
$(-\jj\sigma_2\bmK)\e_1=\e_2$
and
$(-\jj\sigma_2\bmK)\e_2=-\e_1$, we have:
\begin{align}\label{c2-2}
\jj\sigma_r
(\bar P_{\ell,m}+\bar Q_{\ell,m}\sigma_r\varSigma)\hh_{\ell,-m}\e_2
-
\epsilon(A_{\ell,m}+B_{\ell,m}\sigma_r\varSigma)
\hh_{\ell,m}\e_1
-\jj\sigma_r
\bar S_{\ell}\hh_{\ell,\ell}\e_1=0.
\end{align}
We note that for each set
$A_{\ell,m}$ and $B_{\ell,m}$
there is a choice of
$P_{\ell,m}$, $Q_{\ell,m}$, and $S_{\ell}$
such that \eqref{c2-2} is satisfied
(cf. Lemma~\ref{lemma-all-1}).
The inspection shows that $S_{\ell}=0$.
Indeed, multiplying \eqref{c2-2} by $-\jj\sigma_r$
gives:
\begin{align}\label{c2-3}
(\bar P_{\ell,m}+\bar Q_{\ell,m}\sigma_r\varSigma)\hh_{\ell,-m}\e_2
+\jj\epsilon\sigma_r
(A_{\ell,m}+B_{\ell,m}\sigma_r\varSigma)
\hh_{\ell,m}\e_1
-
\bar S_{\ell}\hh_{\ell,\ell}\e_1=0;
\end{align}
\nb{I cannot understand the following :(}
in the first term
in \eqref{c2-3}, the contribution into the first
component ($\e_1$) contains
$e^{-\jj\phi}\sin\theta\hh_{\ell,-m}$,
so the magnetic quantum number
$-m\le\ell$
(the inequality due to the presence
of $\hh_{\ell,-m}$)
is always less than azimuthal one, $\ell+1$;
the second term contributes
$\cos\theta\hh_{\ell,-m}$,
so $-m\le\ell$
and the magnetic quantum number
$m-1\le\ell-1$
is always less than azimuthal one, $\ell+1$.
At the same time, the magnetic and azimuthal quantum
numbers for the last term are the same.
Using \eqref{c2-2}
to substitute the expression for
$\jj\sigma_r(\bar P_{\ell,m}-\bar Q_{\ell,m}\sigma_r\varSigma)\hh_{\ell,-m}\e_2$
into \eqref{c1} yields:
\begin{align*}
(A_{\ell,m}+B_{\ell,m}\sigma_r\varSigma)\hh_{\ell,m}\e_1
-
\epsilon^2(A_{\ell,m}+B_{\ell,m}\sigma_r\varSigma)
\hh_{\ell,m}\e_1
+R_{\ell}\hh_{\ell,-\ell}\e_2
=\varrho_1,
\end{align*}
which we rewrite as
$
(1-\epsilon^2)
(A_{\ell,m}+B_{\ell,m}\sigma_r\varSigma)\hh_{\ell,m}\e_1
+R_{\ell}\hh_{\ell,-\ell}\e_2
=\varrho_1$,
or\nb{$k_{\ell,m}$ not defined here!}
\begin{align*}
(1-\epsilon^2)
\Big[
A_{\ell,m}\hh_{\ell,m}\e_1
+
B_{\ell,m}
\Big(
-\frac{m}{r}\hh_{\ell,m}\e_1
+\frac{k_{\ell,m}}{r}\hh_{\ell,m+1}\e_2
\Big)
\Big]
+
R_{\ell}\hh_{\ell,-\ell}\e_2
=\varrho_1,
\end{align*}
with some coefficients $k_{\ell,m}$.
Writing
$\varrho_1=
c_{\ell,m}\hh_{\ell,m}\e_1
+c'_{\ell,m}\hh_{\ell,m}\e_2$,
we arrive at the system
\begin{align*}
\begin{cases}
(1-\epsilon^2)
\Big(A_{\ell,m}-B_{\ell,m}\frac{m}{r}\Big)
=c_{\ell,m},
\\[1ex]
(1-\epsilon^2)B_{\ell,m-1}\frac{k_{\ell,m-1}}{r}
+\delta_{-m,\ell}R_{\ell}=c'_{\ell,m},
\end{cases}
\qquad
\ell\in\N_0,
\quad
m\in\Z,
\quad
\abs{m}\le\ell.
\end{align*}
We note that
$k_{\ell,m-1}=0$ if $m=-\ell$;
this is exactly when the second term
in the right-hand side of the second equation
is nonzero,
so the above system always has a
\ac{unique??
We can only specify
$A_{\ell,\ell}-B_{\ell,\ell}\ell/r$,
not $A_{\ell,\ell}$, $B_{\ell,\ell}\ell$
separately...
\nb{$A_{\ell,-\ell}$, $B_{\ell,-\ell}$, no?}
}
solution.
\nb{How is $R_\ell$ determined?}
\end{proof}

\appendix

\section{Appendix: The spherical harmonics}
\label{sect-harmonics}

Let us remind the construction
of spherical harmonics
from \cite{van-der-waerden-1932}.
Let $Q(x_1,\dots,x_n)$ be a homogeneous harmonic polynomial
of degree $\ell$:
\[
0=\Delta Q=\Delta (r^\ell \hh(\Omega)),
\]
where $\hh$ is a function on $\mathbb{S}^{n-1}$.
One has
$\Delta u=r^{-(n-1)}\p\sb r(r^{n-1}\p\sb r u)
+\frac{1}{r^2}\Delta\sb{\Omega}u$,
where $\Delta\sb{\Omega}$ is the Laplace--Beltrami operator on $\mathbb{S}^{n-1}$,
so
\[
0=\Delta Q=\Delta (r^\ell \hh(\Omega))
=r^{-(n-1)}\p\sb r(r^{n-1}\p\sb r (r^\ell))\hh(\Omega)
+r^{-2}\Delta\sb{\Omega}(r^\ell \hh(\Omega)).
\]
It follows that
$
\sigma(\Delta\sb{\Omega})
=\{-\varkappa_{\ell}\sothat \ell\in\N_0\},
$
with $\varkappa_{\ell}=\ell(\ell+n-2)$.
For the completeness, we recall that the Legendre polynomials
can be defined by
(cf. \cite[\S 4.6.4]{thaller1992dirac})
\[
P_{\ell}(w)
=\frac{1}{2^\ell \ell!}
\Big(\frac{d}{dw}\Big)^\ell (w^2-1)^\ell,
\qquad
\ell\in\N_0,
\]
the associated Legendre polynomials
are defined for $0\le m\le\ell$
as in \cite[Chapter 14]{Schiff1949} by
\[
P_{\ell,m}(w)
=
(1-w^2)^{m/2}\Big(\frac{d}{dw}\Big)^{m}P_{\ell}(w),
\qquad
m\in\N_0,
\quad
0\le m\le \ell,
\]
and then the spherical harmonics are defined by
\[
\hh_{\ell,m}
(\theta,\phi)
=
\sqrt{\frac{2\ell+1}{4\pi}\frac{(\ell-m)!}{(\ell+m)!}}
e^{\jj m\phi}P_{\ell,m}(\cos\theta)
\quad \mbox{for}\quad m\ge 0,
\qquad
\hh_{\ell,-m}(\theta,\phi)
=
\overline{\hh_{\ell,m}(\theta,\phi)},
\]
where
$\phi\in[0,2\pi)$ and $\theta\in[0,\pi]$.
We note that this differs from the definition adopted
in \cite[\S 4.6.4]{thaller1992dirac},
where
\[
P_{\ell}^m=(-1)^m P_{\ell,m}
\quad
\mbox{and}
\quad
Y_{\ell}^m
=
(-1)^m\hh_{\ell,m}
\quad \mbox{for}\quad m\ge 0,
\qquad
Y_{\ell}^{-m}
=(-1)^m \overline{Y_{\ell}^m}.
\]

\section{Appendix:
Properties of the spin-orbit operator}
\label{sect-spin-orbit}

For generality, we consider an arbitrary spatial dimension
$x\in\R^n$, $n\in\N$.
The Dirac matrices are self-adjoint
and chosen so that the Dirac operator
$D_m=-\jj\bm\upalpha\cdot\nabla+\upbeta m$,
with $m>0$,
satisfies
\[
D_m^2=(-\Delta^2+m^2)I_N,
\]
with $I\sb{N}$ is the unit matrix of size $N$.
It then follows that $\upalpha^j$ and $\upbeta$
are of size $N=2^{[(n+1)/2]}$ (or its integral factor)
and satisfy the standard relations
\begin{align}\label{anti-alpha-general}
\upalpha_i\upalpha_j
+\upalpha_j\upalpha_i=2\delta_{i j} I_N,
\qquad
\upalpha_i\upbeta
+\upbeta\upalpha_i=0,
\qquad
1\le i,\,j\le n;
\qquad
\upbeta^2=I_N.
\end{align}
The Dirac conjugate of $\psi\in\C^N$
is denoted by
$\bar\psi=\psi^*\upbeta$,
with $\psi^*$ the hermitian conjugate of $\psi$.
The Dirac matrices could be taken in the form
\begin{equation}\label{upalpha-general}
\upalpha^{j}
=
\begin{bmatrix} 0&\upsigma_j^*
\\ \upsigma_j&0\end{bmatrix},
\qquad
\upbeta
=
\begin{bmatrix} I\sb{N/2}&0 \\ 0&-I\sb{N/2}\end{bmatrix},
\end{equation}
with
$\upsigma_j$ the ``generalized'' 
Pauli matrices,
not necessarily selfadjoint,
which satisfy the relations
\begin{align}\label{anti-sigma-general}
\upsigma_i\upsigma_j^*+\upsigma_j\upsigma_i^*
=2\delta_{i j}I_{N/2},
\qquad
\upsigma_i^*\upsigma_j+\upsigma_j^*\upsigma_i
=2\delta_{i j}I_{N/2},
\qquad
1\le i,\,j\le n;
\end{align}
let us mention that
each of the above relations implies the other one.
For $x\in\R^n\setminus\{0\}$, $r=\abs{x}>0$, let
\[
\upsigma_r=r^{-1}x\cdot\bm\upsigma,
\qquad
\upalpha_r=r^{-1}x\cdot\bm\upalpha.
\]
We introduce
\[
\varSigma=\bm\upsigma\cdot\nabla:\;
L^2(\R^n,\C^{N/2})\to L^2(\R^n,\C^{N/2}),
\qquad
\dom(\varSigma)=H^1(\R^n,\C^{N/2}),
\]
and define the angular part of $\varSigma$ by
\[
\varSigma_\Omega:\,L^2(\mathbb{S}^{n-1},\C^{N/2})
\to L^2(\mathbb{S}^{n-1},\C^{N/2}),
\qquad
\dom(\varSigma_\Omega)=H^1(\mathbb{S}^{n-1},\C^{N/2})
\]
by the relation
\[
\varSigma=\upsigma_r\p_r+r^{-1}\varSigma_\Omega.
\]

\begin{remark}
In three spatial dimensions,
as ``generalized Pauli matrices'' $\upsigma_i$,
$1\le i\le 3$,
one uses the standard Pauli matrices
$\sigma\sb{1}=\begin{bmatrix} 0&1 \\ 1&0\end{bmatrix}$,
$\sigma\sb{2}=\begin{bmatrix} 0&-\jj \\ \jj&0\end{bmatrix}$,
$\sigma\sb{3}=\begin{bmatrix} 1&0 \\ 0&-1\end{bmatrix}$;
in four spatial dimensions,
one can additionally take $\upsigma_4=\jj I_2$,
$\upalpha^4=\begin{bmatrix}0&-\jj I_2\\ \jj I_2&0\end{bmatrix}
=-\jj\upbeta \gamma_5$,
where $\gamma_5=\begin{bmatrix}0&I_2\\I_2&0\end{bmatrix}$.
In spatial dimensions $n\ge 5$,
one proceeds by induction,
constructing Dirac matrices
of size $N\times N$, $N=2^{[(n+1)/2]}$
via \eqref{upalpha}
with
$\upsigma_j$, $1\le j\le n-1$,
the $N/2\times N/2$ Dirac matrices
$\upalpha^j$, $1\le j\le n-2$ and $\upbeta$
of size $N/2$ and with $\upsigma_n=\jj I_{N/2}$.
A particular choice of the Dirac matrices
is irrelevant
\nb{Not exactly!}
in view of the Dirac--Pauli theorem;
see e.g. \cite[Theorem VIII.7]{opus}.
\end{remark}

The standard references for the Dirac equation are
\cite{bethe1977quantum,bjorken1965relativistic,thaller1992dirac};
for the higher dimensional case,
see e.g. \cite[Chapter VIII]{opus}.

\begin{lemma}\label{lemma-sigma}
Let $n\in\N$
Let $\upalpha_i$, $1\le i\le n$,
be the self-adjoint Dirac matrices
of size $N=2^{[(n+1)/2]}$ satisfying
\eqref{anti-alpha-general},
and let $\upsigma_i$, $1\le i\le n$,
be the ``generalized'' Pauli matrices
of size $N/2$
satisfying
\eqref{anti-sigma-general}.
One has:
\begin{align}
\label{lemma-sigma-2}
&
\big\{\upalpha_r,\bm\upalpha\cdot\nabla\big\}
=\Big(2\p_r+\frac{n-1}{r}\Big)I_N,
\\
\label{lemma-sigma-3}
&
\big\{
\upalpha\sb r,\bm\upalpha\cdot\nabla-\upalpha_r\p_r
\big\}
=(n-1)I_N,
\\
\label{lemma-sigma-1a}
&
\upsigma_r\,\bm\upsigma^*\cdot\nabla
+
(\bm\upsigma\cdot\nabla)\circ\upsigma_r^*
=
\Big(2\p_r+\frac{n-1}{r}\Big)I_{N/2},
\\
\label{lemma-sigma-1b}
&
\upsigma_r^*\,\bm\upsigma\cdot\nabla
+
(\bm\upsigma^*\cdot\nabla)\circ\upsigma_r
=\Big(2\p_r+\frac{n-1}{r}\Big)I_{N/2}.
\end{align}
\end{lemma}

\begin{proof}
Using the relations
\eqref{anti-sigma-general},
one arrives at \eqref{lemma-sigma-1a}:
\[
\frac{x^i}{r}\upsigma_i\upsigma_j^*\p_j
+
\upsigma_j\p_j\circ \frac{x^i}{r}\upsigma_i^*
=
\upsigma_i\upsigma_j^*\frac{x^i}{r}\p_j
+
\upsigma_j\upsigma_i^*\frac{x^i}{r}\p_j
+
\upsigma_j\upsigma_i^*\frac{\delta_{ij}}{r}
-
\upsigma_j\upsigma_i^*\frac{x^i x^j}{r^3}
=2\p_r+\frac{n}{r}I_{N/2}-\frac{1}{r}I_{N/2}.
\]
The relation \eqref{lemma-sigma-1b}
is obtained by interchanging $\upsigma_j$ and $\upsigma_j^*$.
The identity \eqref{lemma-sigma-2}
is proved similarly
to \eqref{lemma-sigma-1a}
with $\upalpha$ in place of both $\upsigma_j$ and $\upsigma_j^*$.
The identity \eqref{lemma-sigma-3}
is a consequence of \eqref{lemma-sigma-2}.
\end{proof}

\begin{lemma}\label{lemma-identities}
There is the following identity:
$\ds
\varSigma
\upsigma_r^*
\varSigma
=
2\varSigma\circ\p_r+\frac{2}{r}\upsigma_r\p_r
+\frac{n-3}{r}\varSigma-\upsigma_r\Delta
$, $\ \ r>0$.
\end{lemma}
\begin{proof}
For $r=\abs{x}>0$, one has:
\[
\p_r\circ\varSigma
=\p_r\circ \Big(\upsigma_r\p_r+\frac{\varSigma_\Omega}{r}\Big)
=\Big(\upsigma_r\p_r+\frac{\varSigma_\Omega}{r}\Big)
\circ\p_r
-\frac{\varSigma_\Omega}{r^2}
=\varSigma\circ\p_r
-\frac{\varSigma}{r}+\frac{\upsigma_r}{r}\p_r.
\]
Therefore,
using Lemma~\ref{lemma-sigma},
\[
\varSigma
\upsigma_r^*
\varSigma
=
\Big(2\p_r+\frac{n-1}{r}I-\upsigma_r
(\bm\upsigma^*\cdot\nabla)
\Big)
\varSigma
=
2\varSigma\circ\p_r
-\frac{2}{r}\varSigma
+\frac{2}{r}\upsigma_r\p_r
+
\frac{n-1}{r}\varSigma
-\upsigma_r\Delta.
\qedhere
\]
\end{proof}

Let us define the symmetric operator
\[
p_r=-\jj r^{(1-n)/2}\circ\p\sb r \circ r^{(n-1)/2},
\qquad
D(p_r)=C\sp\infty\sb{\mathrm{comp}}(\R^n\setminus\{0\}),
\]
the matrix
$\upalpha\sb r=r^{-1}\bm\upalpha\cdot x$,
$x\in\R^n\setminus\{0\}$,
and
\begin{align}\label{def-g-s}
G=
-\sum\sb{1\le j<k\le n}
\jj \upalpha^j\upalpha^k(x_j p_k-x_k p_j),
\qquad
S=\frac{n-1}{2}+G.
\end{align}
We note that if $n=1$,
then $G=S=0$; from now on we assume that $n\ge 2$.
One has:
\begin{align}\label{g-is-ss}
-\jj\bm\upalpha\cdot\nabla
=
\bm\upalpha\cdot\bm{p}
=\upalpha_r\Big(p_r+\frac{\jj}{r}S\Big)
=-\jj\upalpha_r\Big(\p\sb r+\frac{n-1}{2r}-\frac{1}{r}S\Big)
=-\jj\upalpha_r\Big(\p\sb r-\frac{1}{r}G\Big).
\end{align}
The operator
\begin{equation}\label{def-s}
S
=-\jj r(\upalpha_r\,\bm\upalpha\cdot\bm{p}-p_r)
=r\p_r-r\upalpha_r\,\bm\upalpha\cdot\nabla
+\frac{n-1}{2}
\end{equation}
is called the spin-orbit operator
(we follow \cite{kalf1999note}).

\begin{remark}
\label{remark-ksl}
In three spatial dimensions,
Dirac used the spin-orbit operator
$K=\upbeta(2\bm{S}\cdot\bm{L}+1)=\upbeta\big(\bm{J}^2-\bm{L}^2+\frac 1 4\big)$
with
$\bm{S}=-\frac{\jj}{4}\bm\upalpha\wedge\bm\upalpha
=\frac{1}{2}\begin{bmatrix}\bm\sigma&0\\0&\bm\sigma
\end{bmatrix}
$
spin angular momentum operator,
$\bm{L}=x\wedge (-\jj\nabla)$ the orbital angular momentum operator,
and $\bm{J}=\bm{L}+\bm{S}$ total angular momentum
(see \cite[\S 6.4.3]{thaller1992dirac});
in this case, the operators $S$ from \eqref{def-g-s} and $K$
are related by $S=\upbeta K$.
\end{remark}


\begin{lemma}\label{lemma-s}
The operator $S$ on $L^2(\mathbb{S}^{n-1},\C^N)$
with domain $\dom(S)=H^1(\mathbb{S}^{n-1},\C^N)$
is selfadjoint,
anticommutes with $\upalpha_r$,
and commutes with $\Delta_\Omega$.
\end{lemma}

\begin{proof}
By Lemma~\ref{lemma-sigma},
$\big\{\upalpha_r,\bm\upalpha\cdot\nabla\}=\big(2\p_r+\frac{n-1}{r}\big)I_N$,
hence
$\big\{\upalpha_r,\bm\upalpha\cdot\nabla-\upalpha_r\p_r\}=\frac{n-1}{r}I_N$
and so
\[
(n-1)I_N
=
\big\{\upalpha\sb r,r(\bm\upalpha\cdot\nabla-\upalpha_r\p_r)\big\}
=
\big\{\upalpha\sb r,-\upalpha\sb r G\big\}
=
\Big\{
\upalpha\sb r,
\upalpha\sb r\frac{n-1}{2}-\upalpha\sb r S
\Big\};
\]
it follows that
$\upalpha_r$ anticommutes with $\upalpha_r S$ and hence also with $S$.

Since $S$ and $\upalpha_r$ anticommute and $\upalpha_r^2=I_N$,
one has:
\begin{align*}
\p_r^2+\frac{n-1}{r}\p_r+\frac{\Delta_\Omega}{r^2}
=\Big(\upalpha_r\Big(\p_r+\frac{n-1}{2r}-\frac{S}{r}\Big)\Big)^2
=\Big(\p_r+\frac{n-1}{2r}+\frac{S}{r}\Big)
\Big(\p_r+\frac{n-1}{2r}-\frac{S}{r}\Big).
\end{align*}
Collecting the terms with $1/r^2$,
one arrives at the identity
\begin{equation}\label{s-delta}
\Delta_\Omega=
\Big(\frac{n-1}{2}\Big)^2-\frac{n-1}{2}
+S-S^2
=
\Big(\frac{n-2}{2}\Big)^2
-\Big(S-\frac{1}{2}\Big)^2,
\end{equation}
which shows that $S$ commutes with $\Delta_\Omega$.

Since $p_r$ is symmetric,
considering the relation
$\bm\upalpha\cdot\bm{p}=\upalpha_r p_r+\frac{\jj}{r}\upalpha_r S$
(cf. \eqref{g-is-ss}) on functions from
$H^1_0(\R_{+})\otimes H^1(\mathbb{S}^{n-1})\otimes\C^N$
shows that $\jj\upalpha_r S$
is selfadjoint
as an operator on $L^2(\mathbb{S}^{n-1},\C^N)$
with domain $H^1(\mathbb{S}^{n-1},\C^N)$,
and hence so is $S$
(due to the relation
$-\jj S^*\upalpha_r^*=(\jj\upalpha_r S)^*=\jj\upalpha_r S
=-\jj S\upalpha_r$).
\end{proof}

\bibliographystyle{alpha}
\bibliography{bibcomech}
\end{document}